\documentclass[10pt]{amsart}

\usepackage{amsfonts,amssymb,amsmath,amscd,amstext}

\usepackage[colorlinks=true,linkcolor=blue,citecolor=blue]{hyperref}
\usepackage[utf8]{inputenc}
\usepackage{microtype}
\usepackage{graphicx}

\renewcommand{\le}{\leqslant}
\renewcommand{\ge}{\geqslant}
\newcommand{\ptl}{\partial}

\newcommand{\rr}{{\mathbb{R}}}
\newcommand{\cc}{{\mathbb{C}}}

\newcommand{\la}{\lambda}

\newcommand{\hh}{{\mathbb{H}}}
\newcommand{\nn}{{\mathbb{N}}}
\newcommand{\sph}{{\mathbb{S}}}

\newcommand{\hhh}{\mathcal{H}}
\newcommand{\lie}{\mathcal{L}}

\newcommand{\sg}{\sigma}
\newcommand{\Om}{\Omega}
\newcommand{\eps}{\varepsilon}

\newcommand{\ga}{\gamma}
\newcommand{\Ga}{\Gamma}

\newcommand{\nuh}{\nu_{h}}

\newcommand{\ntnh}{\frac{\escpr{N,T}}{|N_h|}}
\newcommand{\mnh}{|N_{h}|}

\newcommand{\de}{\delta_{E}}
\newcommand{\dga}{\dot{\ga}}

\newcommand{\escpr}[1]{\langle#1\rangle}
\newcommand{\inv}[1]{#1^\ell}

\DeclareMathOperator{\tor}{Tor}

\DeclareMathOperator{\intt}{int}
\DeclareMathOperator{\reach}{reach}
\DeclareMathOperator{\unp}{Unp}
\DeclareMathOperator{\ttan}{Tan}
\DeclareMathOperator{\ttanh}{Tan_{H}}
\DeclareMathOperator{\nor}{Nor}
\DeclareMathOperator{\norh}{Nor_{H}}

\newtheorem{theorem}{Theorem}[section]
\newtheorem{proposition}[theorem]{Proposition}
\newtheorem{lemma}[theorem]{Lemma}
\newtheorem{corollary}[theorem]{Corollary}

\theoremstyle{definition}

\newtheorem{remark}[theorem]{Remark}
\newtheorem{example}[theorem]{Example}

\theoremstyle{remark}

\newenvironment{enum}{\begin{enumerate}
}{\end{enumerate}}

\numberwithin{equation}{section}

\begin{document}

\title{Tubular neighborhoods in the sub-Riemannian Heisenberg groups}

\author[M.~Ritor\'e]{Manuel Ritor\'e} \address{Departamento de
Geometr\'{\i}a y Topolog\'{\i}a \\
Universidad de Granada \\ E--18071 Granada \\ Espa\~na}
\email{ritore@ugr.es}

\date{\today}

\thanks{Research supported by MEC-Feder grants MTM2007-61919 and  MTM2013-48371-C2-1-P, and Junta de Andalucía grant FQM-325}

\begin{abstract}
We consider the Carnot-Carathéodory distance $\delta_E$ to a closed set $E$ in the sub-Riemannian Heisenberg groups $\hh^n$, $n\ge 1$. The $\hh$-regularity of $\delta_E$ is proved under mild conditions involving a general notion of singular points. In case $E$ is a Euclidean $C^k$ submanifold, $k\ge 2$, we prove that $\delta_E$ is $C^k$ out of the singular set. Explicit expressions for the volume of the tubular neighborhood when the boundary of $E$ is of class $C^2$ are obtained, out of the singular set, in terms of the horizontal principal curvatures of $\ptl E$ and of the function $\escpr{N,T}/|N_h|$ and its tangent derivatives.
\end{abstract}

\maketitle

\thispagestyle{empty}

\bibliographystyle{abbrv} 

\section{Introduction}

In this paper we shall consider tubular neighborhoods of closed sets in the sub-Rie\-ma\-nnian Heisenberg groups $\hh^n$, $n\ge 1$, endowed with their Carnot-Carathéodory distance. We are mainly interested in the regularity of the distance function to a closed set $E$, and in obtaining an expression for the volume of a tubular neighborhood of $E$ in terms of its radius and of the local geometry of $\ptl E$. The corresponding formula in Euclidean space was obtained by Steiner for convex sets \cite{steiner} (see also \cite[\S~4.2]{MR3155183}), by Weyl for smooth submanifolds of Euclidean spaces \cite{MR1507388}, and by Federer for sets of positive reach \cite{MR0110078}. Weyl's result provided the starting point to obtain a generalization of the Gauss\--Bo\-nnet formula, \cite{zbMATH03098391}, \cite{zbMATH03098393}. In all these cases, the volume of a tubular neighborhood is a polynomial whose coefficients, in the smooth case, are integrals of certain scalar functions associated with the Riemmanian curvature tensor. In convex set theory, the coefficients are the well-known Minkowski's \emph{Quermassintegrals} and, in the theory of positive reach sets, the total mass of the curvature measures. Gray's monograph \cite{MR2024928} contains a comprehensive historical account, and generalizations of Weyl's formula to Riemannian manifolds.

In sub-Riemannian manifolds, these problems have been considered in recent works. Arcozzi and Ferrari discussed the case of open sets $\Om\subset\hh^1$ with boundary $S$ and, under the hypothesis of $C^k$ regularity of $S$, $k\ge 2$, they proved that the Carnot-Carathéodory distance function is of class $C^{k-1}$ out of the singular set of points in $S$ where the tangent plane is horizontal, see Theorem~1.1 in \cite{MR2299576}. The same authors later studied the Hessian of this distance function in \cite{MR2386836}. A Steiner type formula has been recently obtained in $\hh^1$ by Balogh et al. The authors proved in Theorem~1.1 in \cite{MR3388879} that the volume of a tubular neighborhood of a domain with $C^\infty$ boundary has, out of the singular set, a power series expansion whose coefficients are integrals of polynomials of certain second order derivatives of the distance function. The proof of this result was obtained by taking iterated divergences of the distance function. A similar formula was obtained by Ferrari \cite{MR2384641} by computing the flow of the horizontal gradient of the distance function. Properties of the Carnot-Carathéodory distance in \textit{special} $2$-step Carnot groups, focusing on the case of $2$-step ones, can be found in  Arcozzi et al.~\cite{afm-17}. Amongst several remarkable results, the authors give a sub-Riemannian version of the Gauss Lemma in Theorem~1.2, and a proof of the $C^k$ regularity, $k\ge 2$, of the distance to a $C^k$ hypersurface without singular points in Theorem~1.1. 
Interesting results on the distance function to curves and surfaces in $\hh^1$, with applications, have been obtained by Arcozzi \cite{arcozzi-talk}. The results in \cite{MR2299576} and \cite{MR2386836} have been used by Ferrari and Valdinoci \cite[\S~2]{MR2461257} to obtain geometric inequalities in $\hh^1$.
Several recent works use the distance function and techniques of Integral Geometry to obtain geometric inequalities in sub-Riemannian manifolds (e.g., \cite{MR2548248}, \cite{1509.00950}, \cite{huang-agms}). From the Brunn-Minkowski type inequality obtained by Leonardi and Masnou \cite{MR2177813}, lower estimates of the volume of a tubular neighborhood of a given set can be obtained. 

In this work, we deal with properties of the Carnot-Carathéodory distance to a closed set $E$ in the sub-Riemannian Heisenberg groups $\hh^n$, $n\ge 1$. The cases where $E$ has $C^2$ boundary and when $E$ is an $m$-dimensional submanifold of class $C^2$ of $\hh^n$ will be specially considered.

The paper has been organized into several sections. In section~\ref{sec:preliminaries} we fix notation and recall basic facts and geometric properties of the Heisenberg groups $\hh^n$, including geodesics, Jacobi fields, basic properties of the Carnot-Carátheodory distance and the horizontal second fundamental form. Many of the results included in this section are known while others have never explicitly appear in the literature. 

In section~\ref{sec:distance} we look at basic properties of the distance function $\delta_E$ to a closed set $E\subset\hh^n$, focusing on the behavior of length-minimizing geodesics. It is a trivial fact that $\delta_E$ is lipschitz with respect to the Carnot-Carathéodory distance and so differentiable almost everywhere by Pansu-Rademacher's Theorem \cite{MR979599}. In the first part of this section, following Federer \cite{MR0110078}, we define a \emph{tangent cone} $\ttan(E,p)$ at a point $p\in E$ roughly as the set of tangent vectors to curves starting at $p$ and contained in $E$, and a \emph{horizontal normal cone} $\norh(E,p)$ as the set of horizontal vectors orthogonal to $\ttan(E,p)$. This notion is independent from the one of tangent cone of a finite perimeter set given in $\hh^n$, see \cite{MR1871966}. Another relevant notion here is that of \emph{singular} point $p\in E$, i.e., one for which the tangent cone $\ttan(E,p)$ is contained in one of the half-spaces of $T_p\hh^n$ determined by the horizontal distribution hyperplane $\hhh_p$. The set of singular points will be denoted by $E_0$. We shall say that a point is \emph{regular} if it is not singular, and we also define the reach of $E$ at a given point $p\in E$, the metric projection $\xi_E$ to $E$ and the set of points $\unp(E)$ with unique metric projection. Standard properties such as the continuity of $\xi_E$ on $\unp(E)$, Proposition~\ref{prop:xicont}, the continuity of the curvature of length-minimizing geodesics on $\unp(E)\setminus E$, Proposition~\ref{prop:lacont}, and the continuity of the initial speed of length-minimizing geodesics on $\unp(E)\setminus E$ for regular points, Proposition~\ref{prop:vcont}, are obtained. In Lemma~\ref{lem:federer} we show that the distance function to a closed set is $\hh$-differentiable in the interior of $(\unp(E)\setminus E)\cap\xi_E^{-1}(\ptl E\setminus E_0)$. Hence the Carnot-Carathéodory distance $\delta_E$ to $E$ is $\hh$-differentiable in the interior of the set of points with unique metric projection to a regular point in $E$.

Given a closed set $E\subset\hh^n$ and a point $q\not\in\hh^n$, there always exists a length-minimizing  geodesic, connecting a point in $\ptl E$ to $q$, realizing the distance from $E$ to $q$. In the second part of section~\ref{sec:distance} we analyze the behavior of such geodesics. We prove in Lemma~\ref{lem:mingeo} that the initial speed of length-minimizing geodesics lies in the horizontal normal cone to the set, and that the curvature of the geodesic lies in a precisely described interval of real numbers. In particular, when the boundary of the closed set $E$ is a $C^1$ hypersurface, we can prove in Theorem~\ref{thm:exp} that the points in $\ptl E$ at minimum distance are regular points where the tangent hyperplane is not horizontal, that the initial speed of a length-minimizing geodesic joining $E$ to a given point is the outer horizontal unit normal to $E$, and that the curvature of the geodesic is exactly
\begin{equation}
\label{*}
\tag{*}\la=2\frac{\escpr{N,T}}{|N_h|},
\end{equation}
where $N$ is the outer unit normal to $\ptl E$, $T$ is the Reeb vector field on $\hh^n$, and $N_h$ is the orthogonal projection of $N$ to the horizontal distribution $\hhh$. The significance of the quantity $\la$ in \eqref{*} for surfaces in $\hh^1$ was recognized by Arcozzi and Ferrari, who called it the \emph{imaginary curvature} of $S$, see \S~1 in \cite{MR2386836}. From Theorem~\ref{thm:exp} we deduce that length-minimizing geodesics leaving $E$ begin at regular points and are unique. This allows to define an exponential map and to precisely describe the regularity of this map and of the distance function. Since $\la$ goes to $\infty$ when we approach the singular set,  the distance where the geodesics are length-minimizing become very small, so that the reach of the set $E$ at $p\in\ptl E$ approaches $0$ when $p$ approaches the singular set, see Corollary~\ref{cor:reach0}. When the boundary of $E$ is merely $\hh$-regular in the sense of Franchi, Serapioni and Serra-Cassano \cite{MR1871966} we can prove that the initial speed of a length-minimizing geodesic is the outer horizontal unit normal, but we don't get additional information on the curvature of the geodesic, see Theorem~\ref{thm:ch1exp}. Finally, in Examples~\ref{ex:vertplane}, \ref{ex:isolated} and \ref{ex:curve} we analyze the behavior of length-minimizing geodesics in vertical planes and near isolated singular points and singular curves in particular examples. The last two examples should be compared to the results by Arcozzi and Ferrari in \cite[\S~3]{MR2299576}.

In section~\ref{sec:distsm} we treat the regularity of the distance function to a given $m$-dimen\-sional submanifold $S$ of class $C^k$, $k\ge 2$, in $\hh^n$. Here the tangent $\ttan(S,p)$ coincides with the classical tangent space $T_pS$ of the submanifold $S$, and the set of singular points $S_0\subset S$ is composed of those $p\in S$ for which $T_pS\subset\hhh_p$. As a consequence of the techniques developed in section~\ref{sec:distance}, we are able to prove that length-minimizing geodesics leaving $S$ from a regular point have a unique geodesic curvature, which also allows us to define an exponential map. Our main result in this section, Theorem~\ref{thm:tube}, is that the reach of $S$ is positive on compact subsets $K$ of $S\setminus S_0$, and that the distance function $\delta_S$ is of class $C^k$ near $S$ on $\xi^{-1}_S(K)$ when $S$ is of class $C^k$, $k\ge 2$. A corresponding result, Theorem~\ref{thm:reachc11}, is proved when $S$ is a hypersurface of class $C^{1,1}$, generalizing a result by Arcozzi and Ferrari in $\hh^1$, \cite{MR2299576}. In particular, it is obtained in Proposition~\ref{prop:parallelc11} that the parallel hypersurfaces are of class $C^{1,1}$.

Finally, in section~\ref{sec:steiner}, we obtain a Steiner type formula for the volume of the tubular neighborhood of a set with $C^2$ boundary in $\hh^n$. To obtain this formula we follow a classical approach, using Jacobi fields associated to variations by length-minimizing geodesics to compute the volume element along a variation by parallels, and using a coarea formula. In the case of $\hh^1$ we get in  Theorem~\ref{thm:steiner1} the following explicit formula for the tubular neighborhood $U_r$ of radius $r>0$ of points whose metric projection lies in an open subset $U\subset S$ such that $\overline{U}\subset S\setminus S_0$:
\begin{equation*}
|U_r|= \sum_{i=0}^4\int_U\bigg\{\int_0^r a_i f_i(\la, s)\,ds\bigg\}\,dS.
\end{equation*}
Here $\la$ is the function $2\escpr{N,T}/|N_h|$ defined in \eqref{*}, the functions $f_i$ are explicit trigonometric analytic functions defined in \eqref{eq:fi} and \eqref{eq:Fi}, $dS$ is the Riemannian area element associated to the canonical left-invariant Riemannian metric in $\hh^n$, and the coefficients $a_i$ are given by the expressions
\begin{equation*}
\begin{split}
a_{0}&=|N_{h}|,
\\
a_{1}&=|N_h| H,
\\
a_{2}&=-4|N_h|e_1\bigg(\ntnh\bigg),
\\
a_{3}&=-4e_2\bigg(\ntnh\bigg),
\\
a_{4}&=-4He_2\bigg(\ntnh\bigg)-4|N_h|\bigg(e_1\bigg(\ntnh\bigg)\bigg)^2,
\end{split}
\end{equation*}
where $H$ is the sub-Riemannian mean curvature of $S\setminus S_0$, $e_1=J(\nuh)$, where $\nuh=N_h/|N_h|$ is the horizontal unit normal to $S$, $J$ is the standard 90 degrees horizontal rotation, to be defined in \S~\ref{sub:heisenberg}, and $e_2=\escpr{N,T}\nuh-\mnh T$. It is worth noting that all these coefficients depend on the local geometry of the surface $S$. A similar formula has been obtained by Balogh et al., see Theorem~1.1 in \cite{MR3388879}. The quantity $|U_r|$ can be developed as a power series in $r$. The expression up to order three is
\begin{multline*}
|U_r|=A(U)\,r+\frac{1}{2}\bigg(\int_UHdP\bigg)r^2
\\
-\frac{2}{3}\bigg(\int_U\bigg\{e_1\bigg(\ntnh\bigg)+\bigg(\ntnh\bigg)^2\bigg\}dP\bigg)r^3+o(r^4),
\end{multline*}
where $U$ is an open set in $S$ with $\overline{U}\subset S\setminus S_0$. The quantity $A(U)$ is the sub-Riemannian area of $U$, $dP$ is the sub-Riemannian area element on $S$, $H$ is the sub-Riemannian mean curvature of $S$ computed as the sum of the principal curvatures of the horizontal second fundamental form defined in \S~\ref{sec:hor2nd}, the vector $e_1$ is equal to $J(\nuh)$, where $\nuh$ is the horizontal unit normal to $S$.

In the case of $\hh^n$, $n\ge 2$, we obtain the formula as the integral of the modulus of the determinant of a computable square $(2n)$ matrix, whose coefficients depend on the geometry of the boundary hypersurface. We obtain in equation \eqref{eq:steiner23} in Theorem~\ref{thm:steiner23} that
\begin{equation*}
\begin{split}
|U_r|&=A(U)r+\frac{1}{2}\bigg(\int_U HdP\bigg)r^2
\\
&-\frac{1}{6}\bigg(\int_U\bigg(4e_1\bigg(\ntnh\bigg)+(2n+2)\bigg(\ntnh\bigg)^2+|\sg|^2-H^2\bigg)dP\bigg)r^3,
\\
&+o(r^4),
\end{split}
\end{equation*}
where we are using the same notation as above. In addition, $|\sg|^2$ is the squared norm of the horizontal second fundamental form of $S$, defined in \S~\ref{sec:hor2nd}.

We conclude the paper by showing how Steiner's formula looks like when $S$ is a umbilic hypersurface in $\hh^n$, a class recently introduced by Cheng et al.~in \cite{cchy}.

The author wishes to thank Sebastiano Nicolussi Golo for his careful reading of the first version of this manuscript, and to both referees for their useful comments.

\section{Preliminaries}
\label{sec:preliminaries}

\subsection{The Heisenberg group}
\label{sub:heisenberg}

The Heisenberg group $\hh^n$ is the $(2n+1)$-dimensional space $\rr^{2n+1}$, endowed with the group law $*$ given by
\[
(z,t)*(w,s)=(z+w, t+s+\sum_{i=1}^n \text{Im}(z_i\bar{w}_i)),
\]
for $(z,t)$, $(w,s)\in\cc^n\times\rr\equiv\rr^{2n+1}$. In $\hh^n$ we may consider the contact $1$-form
\[
\theta:=dt+\sum_{i=1}^n (-y_{i}\,dx_{i}+x_{i}\,dy_{i}),
\]
which satisfies
\[
d\theta=\sum_{i=1}^n2\,dx_{i}\wedge dy_{i},
\]
and the \emph{horizontal distribution} $\hhh:=\ker(\theta)$ generated by the left-invariant vector fields
\[
X_{i}:=\frac{\ptl}{\ptl x_{i}}+y_{i}\,\frac{\ptl}{\ptl t}, \qquad
Y_{i}:=\frac{\ptl}{\ptl y_{i}}-x_{i}\,\frac{\ptl}{\ptl t}, \qquad
i=1,\ldots,n.
\]
A vector field is \emph{horizontal} if it is tangent to the horizontal distribution at every point. We shall say that a $C^1$ curve $\ga:I\to\hh^n$ is \emph{horizontal} if the tangent vector $\dot{\ga}(t)$ belongs to $\hhh_{\ga(t)}$ for any $t\in I$. A basis of left-invariant vector fields is given by
\begin{equation}
\label{eq:basis}
\{X_1,...,X_n,Y_1,...,Y_n,T\},
\end{equation}
where
\[
T:=\frac{\ptl}{\ptl t}
\]
is the Reeb vector field of the contact manifold $(\hh^n,\theta)$: the only vector field such that $\theta(T)=1$ and $\mathcal{L}_T\theta=0$, where $\mathcal{L}$ denotes the Lie derivative in $\hh^n$. Any left-invariant vector field is a linear combination (with constant coefficients) of the ones in \eqref{eq:basis}. The only non-trivial bracket relation between the vector fields in \eqref{eq:basis} is
\begin{equation}
\label{eq:liebracket}
[X_i,Y_i]=-2T, \qquad i=1,\ldots,n.
\end{equation}
Since $d\theta(X,Y)=X(\theta(Y))-Y(\theta(X))-\theta([X,Y])$, condition \eqref{eq:liebracket} implies that the distribution $\hhh$ is completely non-integrable by Frobenius Theorem. A field of endomorphisms $J:\hhh\to\hhh$ such that $J^2=-\text{Id}$ is defined by $J(X_i):=Y_i$, $J(Y_i):=-X_i$, $i=1,\ldots,n$. We extend it to the whole tangent space by setting $J(T):=0$.

We choose on $\hh^n$ the Riemannian metric $\escpr{\cdot,\cdot}$  so that the basis $\{X_{i}, Y_{i}, T: i=1,\ldots,n\}$ is orthonormal. The norm of a vector field $X$ with respect to this metric will be denoted by $|X|$, and the associated Levi-Civita connection by $D$. If $X$ is any vector field, we shall denote by $X_h:=X-\escpr{X,T}\,T$ its orthogonal projection to the horizontal distribution. Since $T$ is orthogonal to $\hhh$ and $\theta(T)=1$, we infer that $\theta(X)=\escpr{X,T}$ for any vector field $X$. Writing any pair of left-invariant vector fields $X$, $Y$ as a linear combination (with constant coefficients) of the elements of the basis \eqref{eq:basis} and using \eqref{eq:liebracket} we have
\begin{equation}
\label{eq:bracket2}
[X,Y]=2\,\escpr{X,J(Y)}\,T.
\end{equation}
In particular this implies
\[
(\lie_{T}J)(X)=[T,J(X)]-J([T,X])=0.
\]
Moreover, if $X$, $Y$ are left-invariant and horizontal, we get
\[
d\theta(X,Y)=-\theta([X,Y])=-2\,\escpr{X,J(Y)},
\]
which implies that the quadratic form
\[
X\in\hhh\mapsto d\theta(X,J(X))=2\,|X|^2
\]
is positive definite.
Let $\nabla$ be any affine connection in $\hh^n$ with torsion tensor $\tor(X,Y):=\nabla_XY-\nabla_YX-[X,Y]$. Assuming it is a metric connection with respect to the the left-invariant Riemannian metric $\escpr{\cdot,\cdot}$ previously defined, we have
\begin{equation}
\label{eq:metric}
X\,\escpr{Y,Z}=\escpr{\nabla_XY,Z}+\escpr{Y,\nabla_XZ}.
\end{equation}
Then $\nabla$ can be computed in terms of the scalar product and the torsion tensor using Koszul formula
\begin{align*}
2\,\escpr{\nabla_XY,Z}=X\escpr{Y,Z}&+Y\escpr{X,Z}-Z\escpr{X,Y}
\\
&+\escpr{[X,Y],Z}-\escpr{[Y,Z],X}-\escpr{[X,Z],Y}
\\
&+\escpr{\tor(X,Y),Z}-\escpr{\tor(Y,Z),X}-\escpr{\tor(X,Z),Y}.
\end{align*}
Since the scalar product of left-invariant vector fields is a constant function, the Levi-Civita connection $D$ is torsion-free, and \eqref{eq:bracket2}, we obtain, for any pair of left-invariant vector fields $X$, $Y$,
\begin{equation}
\label{eq:levicivita}
D_XY=\theta(Y)\,J(X)+\theta(X)\,J(Y)-\escpr{X,J(Y)}\,T.
\end{equation}
Observe that, in particular, for $X$ left-invariant,
\[
D_XT=J(X),\qquad D_XX=0.
\]

The \emph{pseudo-hermitian connection} $\nabla$ in $\hh^n$ is the only metric connection whose torsion tensor $\tor$ satisfies 
\begin{equation}
\label{eq:torsion}
\tor(X,Y)=-2\,\escpr{X,J(Y)}\,T,
\end{equation}
for any pair of arbitrary vector fields $X$, $Y$, see \cite{MR1000553} and \cite{MR1267892}. Equation \eqref{eq:bracket2} then implies
\[
\tor(X,Y)=-[X,Y],
\]
for any pair of left-invariant vector fields $X$, $Y$. From Koszul formula we get
\begin{equation}
\label{eq:parallel}
\nabla_X Y=0,
\end{equation}
for any pair of left-invariant vector fields $X$, $Y$.

The pseudo-hermitian connection and the Levi-Civita connection can be related by Koszul formula to get
\[
2\,\escpr{\nabla_{X}Y,Z}=2\,\escpr{D_{X}Y,Z}+\escpr{\tor(X,Y),Z}
-\escpr{\tor(X,Z),Y}-\escpr{\tor(Y,Z),X}.
\]

If $R$ is the curvature operator associated to the pseudo-hermitian connection $\nabla$, equation \eqref{eq:parallel} implies
\begin{equation}
R(X,Y)Z=\nabla_X\nabla_YZ-\nabla_Y\nabla_XZ-\nabla_{[X,Y]}Z=0
\end{equation}
for $X$, $Y$, $Z$ left-invariant vector fields. This implies that the connection $\nabla$ is flat.

If $X$, $Y$ are left-invariant, then $J(Y)$ is also left-invariant and
\[
\nabla_X J(Y)-J(\nabla_XY)=0.
\]
This implies $\nabla J=0$ ($J$ is integrable in the sense of Frobenius).

\subsection{Horizontal curves and geodesics in $\hh^n$}
\label{sub:geodesics}

We refer the reader to \cite[\S~3]{MR2435652} for detailed arguments.  A \emph{smooth geodesic} in $\hh^n$ is a smooth horizontal curve $\ga:I\to\hh^n$ which is a critical point of the Riemannian length $L(\ga):=\int_{I}|\dot{\ga}(s)|\,ds$ for any variation by horizontal curves $\ga_{\eps}$ with fixed endpoints. Consider a variation $\{\ga_u\}$, $|u|\le \eps$, of $\ga=\ga_0$ with variational vector field $U:=\ptl\ga_u/\ptl u$. The variation of $\escpr{\dot{\ga}_u,T}$ in the direction of $U$ was computed in \cite{MR2271950} and is given by $U\,\escpr{\dot{\ga}_u,T}=\dot{\ga}\,\escpr{U,T}+2\,\escpr{\dot{\ga},J(U)}$. Hence if $\{\ga_u\}$ are horizontal curves then $U$ satisfies the equation
\begin{equation}
\label{eq:admissible}
\dot{\ga}\,\escpr{U,T}+2\,\escpr{\dot{\ga},J(U)}=0.
\end{equation}
Conversely, if $U$ satisfies equation \eqref{eq:admissible} we choose a vector field $V$ along $\ga$ so that $\dot{\ga}\,\escpr{U,T}+2\,\escpr{\dot{\ga},J(U)}\neq 0$ (for instance $V(s):=sT_{\ga(s)}$). We consider the variation $F(s,u,v):=\exp_{\ga(s)}(uU_{\ga(s)}+vV_{\ga(s)})$, where $\exp$ is the exponential map with respect to the Riemannian metric $\escpr{\cdot,\cdot}$, and the function
\[
f(s,u,v):=\escpr{\frac{\ptl F}{\ptl s}(s,u,v),T_{F(s,u,v)}}.
\]
Then we have $f(s,0,0)=0$, $\tfrac{\ptl f}{\ptl u}(s,0,0)=0$, $\tfrac{\ptl f}{\ptl v}(s,0,0)\neq 0$. By the Implicit Function Theorem, there exists $v(s,u)$ with $v(s,0)=0$, such that $s\mapsto F(s,u,v(s,u))$ is a horizontal curve for $u$ small. Moreover, since $\ptl v/\ptl u=0$ by the Implicit Function Theorem we have that the associated variational vector field is $U$.

So assume that $\ga:I\to\hh^n$ is a smooth regular ($\dot{\ga}\neq 0$) horizontal curve reparameterized to have constant speed ($|\dga|=c\in\rr\setminus\{0\}$). The derivative of the length for a variation of $\ga$ by horizontal curves is given by
\begin{equation}
\label{eq:1stlength}
\frac{d}{d\eps}\bigg|_{\eps=0} 
L(\ga_{\eps})=-\int_{I}\escpr{\nabla_{\dot{\ga}}\dot{\ga},U},
\end{equation}
where by $\nabla_{\dot{\ga}}V$, with $V$ vector field along $\ga$, we have denoted the covariant derivative of $V$ along $\ga$ with respect to the pseudo-hermitian connection $\nabla$. Observe that $\nabla_{\dot{\ga}}\dot{\ga}$ is orthogonal to both $\dot{\ga}$ (since $|\dga|$ is constant) and $T$ (since $T$ is parallel for $\nabla$).  Consider an orthonormal basis of $T\hh^n$ along $\ga$ given by $T$, $\dot{\ga}$, $J(\dot{\ga})$, $Z_{1}$, $\ldots$, $Z_{2n-2}$.  As in the case of the first Heisenberg group $\hh^1$ \cite[\S~3]{MR2271950}, we take any smooth function $f:I\to\rr$ vanishing at the endpoints of $I$ and such that $\int_{I}f=0$.  Then the vector field $U$ along $\ga$ defined by the conditions $U_{h}=fJ(\dot{\ga})$ and $\escpr{U,T}=2\,\int_{I} f$ satisfies \eqref{eq:admissible}.  From \eqref{eq:1stlength} we conclude that $\escpr{D_{\dot{\ga}}\dot{\ga},J(\dot{\ga})}$ is constant.  Now let $g:I\to\rr$ be any smooth function vanishing at the endpoints of $I$.  Then the vector field $U=gZ_{i}$, for $i=1,\ldots,2n-2$, satisfies \eqref{eq:admissible}, and hence $D_{\dot{\ga}}\dot{\ga}$ is orthogonal to $Z_{i}$ for all $i=1,\ldots,2n-2$.  So we obtain that the horizontal geodesic $\ga:I\to\hh^n$ satisfies the equation
\begin{equation}
\label{eq:geodesic}
\nabla_{\dot{\ga}}\dot{\ga}+\la\,J(\dot{\ga})=0,
\end{equation}
for some constant $\la\in\rr$.  If $\ga$ satisfies \eqref{eq:geodesic} we shall say that $\ga$ is a smooth geodesic of \emph{curvature} $\la/|\dga|$. Observe that any curve satisfying \eqref{eq:geodesic} has constant speed since 
\[
\dga|\dga|^2=2\escpr{\nabla_{\dga}\dga,\dga}=-2\la\escpr{J(\dga),\dga}=0.
\]
Moreover, the notion of curvature of a smooth geodesic is invariant by reparameterization with constant speed: in case $\ga$ satisfies \eqref{eq:geodesic} and $c\in\rr$ is different from $0$, we define $\ga_c(s):=\ga(cs)$. Then we have
\[
\nabla_{\dga_c}{\dga_c}=c^2\nabla_{\dga} \dga=-\la c^2J(\dga)=-\la cJ(\dga_c),
\]
that has curvature $\la c/|\dga_c|=\la/|\dga|$ as claimed. If $\ga$ is parameterized by arc-length and satisfies \eqref{eq:geodesic} then $\ga_c$ is a smooth geodesic with constant speed and curvature $\la$.

For $\la\in\rr$, $p\in\hh^n$, and $v\in T_{p}\hh^n$, $v\neq 0$, we shall denote by
\begin{equation}
\ga_{p,v}^\la
\end{equation}
the geodesic $\ga:\rr\to\hh^n$ satisfying \eqref{eq:geodesic} with initial conditions $\ga(0)=p$, $\dot{\ga}(0)=v$.

The curve $\ga_{p,v}^\la$ has curvature $\la/|v|$. Let $\ga=\ga_{p,v}^\la$. If $c\neq 0$ then $\ga_c$ is a geodesic satisfying \eqref{eq:geodesic} with constant $c\la$ and initial conditions $\ga_c(0)=p$ and $\dga_c(0)=c\dga(0)=cv$. Hence we have $\ga_c(s)=\ga_{p,cv}^{c\la}(s)$ and so
\begin{equation}
\label{eq:geodrelation}
\ga_{p,v}^\la(cs)=\ga_{p,cv}^{c\la}(s).
\end{equation}

The equations of a geodesic can be computed easily: let 
\[
\ga(s)=(x_{1}(s), y_{1}(s), \ldots, x_{n}(s), y_{n}(s), t(s))
\]
be a 
horizontal geodesic. Then
\begin{align*}
\dot{\ga}(s)&=\sum_{i=1}^n \dot{x}_{i}(s)\,(X_{i})_{\ga(s)}
+\dot{y}_{i}(s)\,(Y_{i})_{\ga(s)},
\\
\dot{t}(s)&=\sum_{i=1}^n (\dot{x}_{i} y_{i}-x_{i}\dot{y}_{i})(s).
\end{align*}
From \eqref{eq:geodesic}, the coordinates of $\ga$ satisfy the system 
\begin{align*}
\ddot{x}_{i}&=\la\,\dot{y}_{i},
\\
\ddot{y}_{i}&=-\la\,\dot{x}_{i},
\end{align*}
with initial conditions $x_{i}(0)=(x_{0})_{i}$,
$y_{i}(0)=(y_{0})_{i}$, $\dot{x}_{i}(0)=A_{i}$,
$\dot{y}_{i}(0)=B_{i}$, and $\sum_{i=1}^n (A_{i}^2+B_{i}^2)=|\dga(0)|^2$.

Integrating these equations, for $\la=0$, we obtain
\begin{align*}
x_{i}(s)&=(x_{0})_{i}+A_{i}s, \\
y_{i}(s)&=(y_{0})_{i}+B_{i}s, \\
t(s)&=t_{0}+\sum_{i=1}^n (A_{i}(y_{0})_{i}-B_{i}(x_{0})_{i})\,s,
\end{align*}
which are horizontal Euclidean straigth lines in $\hh^n$.

Integrating, for $\la\neq 0$, we obtain
\begin{equation}
\label{eq:geodesics}
\begin{split}
x_{i}(s)&=(x_{0})_{i}+A_{i}\,\bigg(\frac{\sin(\la s)}{\la}\bigg)
+B_{i}\,\bigg(\frac{1-\cos(\la s)}{\la}\bigg),
\\
y_{i}(s)&=(y_{0})_{i}-A_{i}\,\bigg(\frac{1-\cos(\la s)}{\la}\bigg)
+B_{i}\,\bigg(\frac{\sin(\la s)}{\la}\bigg),
\\
t(s)&=t_{0}+\frac{|\dga(0)|^2}{\la}\,\bigg(s-\frac{\sin(\la s)}{\la}\bigg)
\\
&\qquad\qquad+\sum_{i=1}^n\bigg\{
(A_{i}(x_{0})_{i}+B_{i}(y_{0})_{i})\bigg(\frac{1-\cos(\la s)}{\la}\bigg)\bigg\}
\\
&\qquad\qquad
-\sum_{i=1}^n\bigg\{(B_{i}(x_{0})_{i}-A_{i}(y_{0})_{i})
\bigg(\frac{\sin(\la s)}{\la}\bigg)\bigg\}.
\end{split}
\end{equation}

For future reference, we shall consider the analytic functions
\begin{equation}
\label{eq:FGH}
F(x):=\frac{\sin(x)}{x},\quad  G(x):=\frac{1-\cos(x)}{x}, \quad H(x):=\frac{x-\sin(x)}{x^2},
\end{equation}
the analytic functions
\begin{equation}
\label{eq:Fi}
\begin{split}
F_1(x)&:=\frac{\sin(x)}{x},
\\
F_2(x)&:=\frac{1-\cos(x)}{x^2},
\\
F_3(x)&:=\frac{\sin(x)-x\cos(x)}{x^3},
\\
F_4(x)&:=\frac{2-2\cos(x)-x\sin(x)}{x^4},
\\
K(x)&:=\frac{x-\sin(x)}{x^3},
\end{split}
\end{equation}
and the functions
\begin{equation}
\label{eq:fi}
\begin{split}
f_0(\la, s)&:=\cos(\la s),
\\
f_1(\la, s)&:=F_1(\la s)s,
\\
f_2(\la, s)&:=F_2(\la s)s^2,
\\
f_3(\la, s)&:=F_3(\la s)s^3,
\\
f_4(\la, s)&:=F_4(\la s)s^4,
\\
k(\la, s)&:=K(\la s)s^3,
\end{split}
\end{equation}
that are analytic functions of $\la$ and $s$.

Let $\pi:\hh^n\to\rr^{2n}$ be the Riemannian submersion over $\rr^{2n}$.  Fix a point $p\in\hh^n$ and identify a horizontal vector $v\in\mathcal{H}_{p}$ with the vector $w$ in $\rr^{2n}$ given by the coordinates of $v$ in the basis $\{X_{i},Y_{i} : i=1,\ldots,n\}$. Denote by $t$ the Euclidean height function in $\hh^n$.  With this identification, the involution $J$ induces a involution on vectors of $\rr^{2n}$, 
\[
(A_{1},B_{1},\ldots ,A_{n},B_{n})\mapsto (-B_{1},A_{1},\ldots
,-B_{n},A_{n})
\]
that will be also denoted by $J$.

Choose $\la\in\rr$ and consider the geodesic $\ga:=\ga_{p,v}^\la$. Let $\alpha:=\pi\circ\ga$ and $\beta=t\circ\ga$. 
Then we have
\begin{equation}
\label{eq:geodesic2}
\begin{split}
\alpha(s)&=\pi(p)+s\big(F(\la s)\,w-G(\la s)\,J(w)\big) ,
\\
\beta(s)&=t(p)+|\dga(0)|^2s^2H(\la s)+\escpr{\pi(p),s\big(G(\la s)\,w
+F(\la s)\,J(w)\big)},
\end{split}
\end{equation}
where $\escpr{\cdot,\cdot}$ is the Euclidean product in $\rr^{2n}$.

\begin{remark}
If $\Ga:I\to\rr^{2n}$ is a $C^1$ curve and $c\in\rr$, then
\[
s\mapsto \big(\Ga(s),c+\frac{1}{2}\int_0^s\escpr{J(\dot{\Ga}),\Ga}(\xi)\,d\xi\big)
\]
is a smooth horizontal curve in $\hh^n$ with initial condition $(\Ga(0),c)$.
\end{remark}

\subsection{Jacobi fields in $\hh^n$}

Some of the results of this section have already appeared in \cite[\S~6]{MR1267892} and \cite{MR2548248}, see also Cheeger-Ebin's \cite{MR0458335} for the case of Riemannian manifolds.

Let $\ga:I\to\hh^n$ be a smooth geodesic in $\hh^n$ of curvature $\la$ parameterized by arc-length satisfying equation $\nabla_{\dot{\ga}}\dot{\ga}+\la\,J(\dot{\ga})=0$. Consider a variation $\{\ga_\eps\}$ of $\ga=\ga_0$ by curves $\ga_\eps:I\to\hh^n$ satisfying
\[
\nabla_{\dga_\eps}\dga_\eps+\la(\eps)J(\dga_\eps)=0.
\]
We know that the curves $\ga_\eps$ have constant speed and curvature $\la(\eps)/|\dga_{\eps}|$. Let $U:=\ptl\ga_\eps/\ptl\eps|_{\eps=0}$ the deformation vector field. Then we have
\[
\nabla_{U}\nabla_{\dga_{\eps}}\dga_{\eps}+\la'J(\dga_\eps)+\la\,J(\nabla_{\dga_\eps}U)=0,
\]
where $\la'=U(\la)=\ptl \la(\eps)/\ptl\eps|_{\eps=0}$. Using the sub-Riemannian curvature tensor $R$ associated to $\nabla$ we get
\[
\nabla_{U}\nabla_{\dga_{\eps}}\dga_{\eps}=R(U,\dga_{\eps})\,\dga_{\eps}+\nabla_{\dga_{\eps}}\nabla_{U}\dga_{\eps}
+\nabla_{[U,\dga_{\eps}]}\dga_{\eps}=\nabla_{\dga_{\eps}}\nabla_{U}\dga_{\eps},
\]
since $R=0$ in $\hh^n$ and $[U,\dga_{\eps}]=0$. From \eqref{eq:torsion} we have
\begin{equation*}
\nabla_{\dga_{\eps}}\nabla_{U}\dga_{\eps}=\nabla_{\dga_{\eps}}\nabla_{\dga_{\eps}}U
+\nabla_{\dga_{\eps}}\tor(U, \dga_{\eps})=\nabla_{\dga_{\eps}}\nabla_{\dga_{\eps}}U-2\dga_{\eps}\escpr{U,J(\dga_{\eps})}\,T.
\end{equation*}
Evaluating at $\eps=0$ we obtain the Jacobi equation along the geodesic $\ga$. 
\begin{equation*}
\nabla_{\dga}\nabla_{\dga}U+\la J(\nabla_{\dga}U)+\la'
J(\dga)-2\dga\,\escpr{U,J(\dga)}\,T=0.
\end{equation*}
Letting $\dot{U}=\nabla_{\dga}U$, $\ddot{U}=\nabla_{\dga}\nabla_{\dga} U$ we get
\begin{equation}
\label{eq:jacobi}
\ddot{U}+\la J(\dot{U})+\la'
J(\dga)-2\dga\,\escpr{U,J(\dga)}\,T=0.
\end{equation}
A Jacobi field along a sub-Riemannian geodesic $\ga$ will be a vector field along $\ga$ that is a solution of equation \eqref{eq:jacobi}. The solutions of equation \eqref{eq:jacobi} can be explicitly computed in $\hh^n$. We first observe that the tangent vector of a sub-Riemannian geodesic is contained in a horizontal two-dimensional plane determined by its initial velocity.

\begin{lemma}
\label{lem:geocoor}
Let $\ga:\rr\to\hh^n$ be a geodesic in $\hh^n$ with curvature $\la$, and let $X$ be a
left-invariant~vector field. Then
\begin{equation}
\label{eq:gax}
\begin{split}
&\frac{d}{ds}\escpr{\dga,X_{\ga(s)}}=\la\,\escpr{\dga,J(X_{\ga(s)})}, \\
&\frac{d}{ds}\escpr{\dga,J(X_{\ga(s)})}=-\la\,\escpr{\dga,X_{\ga(s)}},
\end{split}
\end{equation}
where $s$ is the arc-length parameter of $\ga$. In particular
\begin{equation}
\label{eq:gaxjx}
\begin{split}
\escpr{\dga(s),X_{\ga(s)}}&=\escpr{\dga(0),X_{\ga(0)}}\cos(\la s)+\escpr{\dga(0),J(X_{\ga(0)})}\sin(\la s),
\\
\escpr{\dga(s),J(X_{\ga(s)})}&=-\escpr{\dga(0),X_{\ga(0)}}\sin(\la s)+\escpr{\dga(0),J(X_{\ga(0)})}\cos(\la s),
\end{split}
\end{equation}
Moreover, if $X$ is a left-invariant vector field so that
$X_{p}=\dga(0)$ then
\[
\dga(s)=\cos(\la s)\,X_{\ga(s)}-\sin(\la s)\,J(X_{\ga(s)}).
\]
\end{lemma}

\begin{proof}
The system \eqref{eq:gax} is obtained from the geodesic equation \eqref{eq:geodesic} taking into account that left-invariant vector fields in $\hh^n$ are parallel for the pseudo-hermitian connection. From this observation, equations \eqref{eq:gaxjx} are obtained. Assume $\la\neq 0$ since the case $\la=0$ is trivial. Then the expression for $\dga$ is obtained by extending $X$ to an orthonormal basis of left-invariant vector fields $X$, $J(X)$, $X_{2}$, $J(X_{2}),\ldots, X_{n}$, $J(X_{n})$, and using   equations \eqref{eq:gax} and \eqref{eq:gaxjx}.
\end{proof}

Now we compute explicitly the Jacobi fields along a given sub-Riemannian geodesic. Let us introduce the following notation: if $v\in T_p\hh^n$, then $v^\ell$ is the only left-invariant vector field such that $(v^\ell)_p=v$. 

\begin{lemma}
\label{lem:jacobi}
Let $\ga:\rr\to\hh^n$ be a sub-Riemannian geodesic of curvature $\la$ parameterized by arc-length, and let $U$ be a Jacobi field along $\ga$ satisfying equation \eqref{eq:jacobi}. Then $U$ is given by $U(s)=U_h(s)+c(s)\,T_{\ga(s)}$,
where $U_h$ and $c$ satisfy the equations:
\begin{equation}
\label{eq:jacobih}
\ddot{U}_h+\la J(\dot{U}_h)+\la'J(\dga)=0,
\end{equation}
and
\begin{equation}
\label{eq:jacobiv}
\dot{c}=2b,\quad \text{where } b=\escpr{U,J(\dga)}.
\end{equation}

Moreover, $U_h$ is given by
\begin{multline}
\label{eq:explicitUh}
U_h(s)=[\inv{U_h(0)}]_{\ga(s)}+f_1(\la, s)\,[\inv{\dot{U}_h(0)}]_{\ga(s)}-\la f_2(\la, s)\,[\inv{J(\dot{U}(0))}]_{\ga(s)}
\\
+\la'\big[\la k(\la, s)\dga(s)+\la f_2(\la, s)J(\dga(s))\big].
\end{multline}
\end{lemma}

\begin{proof}
Consider an orthonormal basis of horizontal left-invariant vector fields $X_1$, $Y_1$, $\ldots$, $X_n$, $Y_n$ so that $Y_i=J(X_i)$ for all $i$. The Jacobi field $U(s)$ can be expressed as
\[
U(s)=U_h(s)+c(s)T_{\ga(s)}=\bigg(\sum_{i=1}^n a_i(s)\,(X_i)_{\ga(s)}+b_i(s)\,(Y_i)_{\ga(s)}\bigg)+c(s)\,T_{\ga(s)}.
\]
Observe that, for any vector field $U$ we have $(\dot{U})_h=\dot{(U_h)}$. Decomposing the Jacobi equation \eqref{eq:jacobi} into their horizontal and vertical components we get
\begin{equation}
\label{eq:jacobidecomp}
\ddot{U}_h+\la J(\dot{U}_h)+\la'J(\dga)=0,\qquad
\ddot{c}=2\dot{b}.
\end{equation}

The first of these equations is exactly \eqref{eq:jacobih}. However, the second one in \eqref{eq:jacobidecomp} is weaker than \eqref{eq:jacobiv}. We notice that $\dot{c}=2b$ is satisfied whenever we have a variation by horizontal curves, since in this case,
\begin{align*}
\dga\escpr{U,T}&=\escpr{\nabla_{\dga} U,T}
\\
&=\escpr{\nabla_U\dga_\eps+\tor(\dga,U),T}
\\
&=U\escpr{\dga_\eps,T}+2\escpr{J(\dga),U}=2\escpr{J(\dga),U},
\end{align*}
as $\escpr{\dga_\eps,T}=0$. Hence the vertical component of the Jacobi equation \eqref{eq:jacobi} does not provide any additional information.

Taking into account the expression for $U_h$, the horizontal Jacobi equation  \eqref{eq:jacobih} implies that the following system
\begin{equation}
\label{eq:1stsystem}
\begin{split}
&\ddot{a}_i-\la\dot{b}_i-\la'\escpr{\dga,Y_i}=0,
\\
&\ddot{b}_i+\la\dot{a}_i+\la'\escpr{\dga,X_i}=0,
\end{split}
\end{equation}
is satisfied by the horizontal components $a_i,b_i$, $i=1,\ldots,n$, of $U$. Defining the coefficients $\alpha_i,\beta_i$ by the equality
\[
\dga(0)=\sum_{i=1}^n \big(\alpha_i\, (X_i)_{\ga(0)}+\beta_i\, (Y_i)_{\ga(0)}\big),
\]
we get $\escpr{\dga,X_i}=\alpha_i\cos(\la s)+\beta_i\sin(\la s)$ and $\escpr{\dga,Y_i}=-\alpha_i\sin(\la s)+\beta_i\cos(\la s)$ from \eqref{eq:gaxjx}. Hence the two equations in \eqref{eq:1stsystem} can be rewritten as
\begin{equation}
\label{eq:2ndsystem}
\begin{split}
&\ddot{a}_i-\la\dot{b}_i+\la'\big(\alpha_i\sin(\la s)-\beta_i\cos(\la s)\big)=0,
\\
&\ddot{b}_i+\la\dot{a}_i+\la'\big(\alpha_i\cos(\la s)+\beta_i\sin(\la s)\big)=0.
\end{split}
\end{equation}
The solutions of this system of ordinary differential equations can be explicitly computed. They are given by
\begin{equation}
\label{eq:a_ib_i}
\begin{split}
a_i(s)&=a_i(0)+\dot{a}_i(0)f_1(\la, s)+\la\dot{b}_i(0)f_2(\la, s)+\la'\alpha_ih(\la, s)+\la'\beta_ij(\la, s),
\\
b_i(s)&=b_i(0)+\dot{b}_i(0)f_1(\la, s)-\la\dot{a}_i(0)f_2(\la, s)-\la'\alpha_ij(\la, s)+\la'\beta_ih(\la, s),
\end{split}
\end{equation}
where $f_1(\la,s)$ and $f_2(\la,s)$ were defined in \eqref{eq:fi}, and $h(\la, s)$ and $j(\la, s)$ are given by
\begin{equation}
\label{eq:fghj}
\begin{split}
h(\la, s)&=\frac{\la s\cos(\la s)-\sin(\la s)}{\la^2},
\\
j(\la, s)&=\frac{-1+\cos(\la s)+\la s\sin(\la s)}{\la^2}.
\end{split}
\end{equation}
They are analytic functions of $\la$ and $s$. In particular, $h(0,s)=0$ and $j(0,s)=s^2/2$ for all $s\in\rr$.

From \eqref{eq:a_ib_i} we get
\begin{equation*}
\begin{split}
U_h(s)&=\sum_{i=1}^n\big(a_i(s) (X_i)_{\ga(s)}+b_i(s) (Y_i)_{\ga(s)}\big)
\\
&=\sum_{i=1}^na_i(0)(X_i)_{\ga(s)}+b_i(0)(Y_i)_{\ga(s)}
\\
&\qquad\qquad+\bigg(\sum_{i=1}^n\dot{a}_i(0)(X_i)_{\ga(s)}+\dot{b}_i(0)(Y_i)_{\ga(s)}\bigg)\,f_1(\la, s)
\\
&\qquad\qquad+\bigg(\sum_{i=1}^n\dot{b}_i(0)(X_i)_{\ga(s)}-\dot{a}_i(0)(Y_i)_{\ga(s)}\bigg)\,\la f_2(\la, s)
\\
&\qquad\qquad+\bigg(\sum_{i=1}^n \alpha_i(X_i)_{\ga(s)}+\beta_i(Y_i)_{\ga(s)}\bigg)\,\la' h(\la, s)
\\
&\qquad\qquad+\bigg(\sum_{i=1}^n\beta_i(X_i)_{\ga(s)}-\alpha_i(Y_i)_{\ga(s)}\bigg)\,\la' j(\la, s).
\end{split}
\end{equation*}
So we have
\begin{multline}
\label{eq:U_hcompact}
U_h(s)=[\inv{U_h(0)}]_{\ga(s)}+f_1(\la, s)\,[\inv{\dot{U}_h(0)}]_{\ga(s)}-\la f_2(\la, s)\,[\inv{J(\dot{U}(0))}]_{\ga(s)}
\\
+\la'h(\la, s)\,[\inv{\dga(0)}]_{\ga(s)}-\la'j(\la, s)\,[\inv{J(\dga(0))}]_{\ga(s)}.
\end{multline}
Letting $X=\dga(0)^\ell$, Lemma~\ref{lem:geocoor} implies
\begin{align*}
X_{\ga(s)}&=\cos(\la s)\dga(s)+\sin(\la s)J(\dga(s)),
\\
J(X)_{\ga(s)}&=-\sin(\la s)\dga(s)+\cos(\la s)J(\dga(s)),
\end{align*}
and so
\[
h(\la, s)X_{\ga(s)}-j(\la, s)J(X)_{\ga(s)}=\la k(\la, s)\dga(s)+\la f_2(\la, s)J(\dga(s)),
\]
that implies \eqref{eq:explicitUh}.
\end{proof}

\begin{remark}
\label{rem:unitspeedjacobi}
If $U$ is a Jacobi field along a geodesic $\ga$ associated to a variation by arc-length parameterized curves, we have
\begin{equation}
\label{eq:dga'J}
\dga\escpr{U,\dga}+\la\escpr{U,J(\dga)}=0
\end{equation}
since
\begin{align*}
0=\tfrac{1}{2} U|\dga|^2&=\escpr{\nabla_{\dga} U+\tor(U,\dga),\dga}
\\
&=\dga\escpr{U,\dga}-\escpr{U,\nabla_{\dga}\dga}=\dga\escpr{U,\dga}+\la\escpr{U,J(\dga)}.
\end{align*}
Moreover, using equations \eqref{eq:jacobiv} and \eqref{eq:dga'J} we get
\begin{equation}
\label{eq:dgat}
\dga\escpr{U,\dga}+\tfrac{\la}{2}\,\dga\escpr{U,T}=0,
\end{equation}
and so $\escpr{U,\dga+\tfrac{\la}{2}T}$ is constant along $\ga$.
\end{remark}

\begin{lemma}
\label{lem:jacobi-2}
Let $\ga:\rr\to\hh^n$ be a sub-Riemannian geodesic of curvature $\la$. Consider a Jacobi field $U$ along $\ga$ satisfying equation \eqref{eq:jacobi} given by
\[
U(s)=a(s)\dga(s)+b(s)J(\dga(s))+c(s)T_{\ga(s)}+\sum_{i=2}^n \big( u_{i}(s)(X_{i})_{\ga(s)}+v_{i}(s)(Y_{i})_{\ga(s)}\big),
\]
where $\{X_{i}, Y_{i}:i=1,\ldots,n\}$ is an orthonormal set of left-invariant horizontal vector fields such that $\dga, J(\dga)$ belong to the space generated by $X_{1}, Y_1$, and $Y_{i}=J(X_{i})$ for all $i$. Then the functions $u_i$, $v_i$, $i\ge 2$, satisfy the equations
\begin{equation}
\label{eq:uivi}
\begin{split}
&\ddot{u}_{i}-\la\,\dot{v}_{i}=0,
\\
&\ddot{v}_{i}+\la\,\dot{u}_{i}=0.
\end{split}
\end{equation}
If we further assume that the variation associated to $U$ consists of arc-length parameterized geodesics, then the functions $a$, $b$, $c$, satisfy the differential equations
\begin{equation}
\label{eq:jacobiequations}
\begin{split}
&\dot{a}+\la b=0,
\\
&\ddot{b}-\la\dot{a}+\la'=0,
\\
&\dot{c}-2b=0.
\end{split}
\end{equation}

\end{lemma}

\begin{proof}
The Jacobi field $U$ can be expressed as a linear combination of $\dga$, $J(\dga)$, $T$ and $X_i$, $Y_i$, $i\ge 2$, by Lemma~\ref{lem:geocoor}, since $\dga$ and $J(\dga)$ are linear combinations of $X_1$ and $Y_1=J(X_1)$. Using the geodesic equation \eqref{eq:geodesic} and the fact that $X_{i}$, $Y_{i}$, $i\ge 2$, and~$T$ are parallel for the pseudo-hermitian connection $\nabla$, the Jacobi equation \eqref{eq:jacobi} can be written as
\begin{align*}
(\ddot{a}+\la\dot{b})\,\dga+(\ddot{b}-\la\dot{a}+\la')\,J(\dga)&+(\ddot{c}-2\dot{b})\,T
\\
&+\bigg(\sum_{i=2}^n (\ddot{u}_{i}-\la \dot{v}_i)\,X_{i}+(\ddot{v}_{i}+\la\dot{u}_i)\,Y_{i}\bigg)
=0.
\end{align*}
This immediately implies \eqref{eq:uivi}. The stronger equation $\dot{a}+\la b=0$ in \eqref{eq:jacobiequations} holds because of \eqref{eq:dga'J} in Remark~\ref{rem:unitspeedjacobi}.

To get the third equation in \eqref{eq:jacobiequations} we differentiate to get
\[
\dot{c}=\dga\,\escpr{U,T}=\escpr{\nabla_{\dga}U,T}=\escpr{\nabla_U\dga+\tor(\dga,U),T}=2\,\escpr{J(\dga),U}=2b.\qedhere
\]

\end{proof}

Taking third order derivatives we immediately see that all components of the Jacobi field satisfy an ordinary differential equation of a given type.

\begin{corollary}
\label{cor:jacobiequations-1}
Let $\ga:\rr\to\hh^n$ a sub-Riemannian geodesic of curvature $\la$, and $U$ a Jacobi field along $\ga$ such that the associated variation consists of unit speed geodesics. Let $c(s)=\escpr{U(s),T_{\ga(s)}}$ and $\la'=U(\la)$. Then
\begin{equation}
\label{eq:odec}
\dddot{c}+\la^2\dot{c}+2\la'=0.
\end{equation}
In particular,
\begin{equation}
\label{eq:expc}
c(s)=c(0)+\dot{c}(0)f_1(\la, s)+\ddot{c}(0)f_2(\la, s)-2\la'k(\la, s),
\end{equation}

where $f_1,f_2$, and $k$ are the functions defined in \eqref{eq:fi}.

\end{corollary}

\begin{proof}
To obtain \eqref{eq:odec} we shall use the equations \eqref{eq:jacobiequations}. We differentiate twice equation $\dot{c}=2b$ to get $\dddot{c}=2\ddot{b}$. Then we replace the value of $\ddot{b}$ to obtain $\dddot{c}=2(\la\dot{a}-\la')$. Finally from \eqref{eq:dgat} we get $2\dot{a}=-\la\dot{c}$. Replacing $2\dot{a}$ in the previous equation we are done.

The expression for $c$ in \eqref{eq:expc} follows from the identities $\tfrac{\ptl}{\ptl s}f_1(\la, s)=\cos(\la s)$, $\tfrac{\ptl}{\ptl s}f_2(\la, s)=f_1(\la, s)$, $\tfrac{\ptl}{\ptl s}k(\la ,s)=f_2(\la, s)$. Thus $f_1(\la,\cdot)$ and $f_2(\la,\cdot)$ satisfy the differential equation $\dddot{u}+\la^2\dot{u}=0$ and $k(\la,\cdot)$ satisfies the equation $\dddot{u}+\la^2\dot{u}=1$.
\end{proof}

We remark that, for $\la=0$, formula \eqref{eq:expc} becomes $c(s)=c(0)+\dot{c}(0)s+\tfrac{1}{2}\ddot{c}(0)s^2-\tfrac{\la'}{3}s^3$ since $F_1(0)=1$, $F_2(0)=1/2$ and $K(0)=1/6$ by \eqref{eq:Fi}.

\subsection{Basic properties of the Carnot-Carath\'eodory distance}

The Carnot-Carath\'eodory distance $d(p,q)$ between $p$, $q\in\hh^n$ is defined as the infimum of the Riemannian length of the piecewise smooth horizontal curves joining $p$ and $q$.  Chow's Theorem \cite{MR1421823} implies that two given points can be joined at least by one piecewise smooth horizontal curve.  Following Gromov \cite[Ch.~1]{MR2307192}, we define a \emph{minimizing geodesic} as an absolutely continuous curve $\alpha:I\to\hh^n$ such that $d(\alpha(s),\alpha(s'))=|s-s'|$.  By the Hopf-Rinow Theorem, \cite{MR2307192}, any pair of points in $\hh^n$ can be joined by a minimizing geodesic.  By Pontryagin Maximum Principle, see \cite{MR1976833}, any minimizing geodesic is a regular smooth horizontal curve satisfying equation \eqref{eq:geodesic} and, hence, it is uniquely determined by its initial conditions $\ga(0)$, $\dot{\ga}(0)$, and its curvature $\la$. We must remark that minimizing geodesics in arbitrary sub-Riemannian manifolds do not need to satisfy the geodesic equations. This leads to the notion of \emph{abnormal geodesics}, see \cite{MR1867362}, \cite{MR2421548}.

Assume $\ga_{1}, \ga_{2}:[0,L]\to\hh^n$ are two different minimizing geodesics joining $p$ and $q$.  Then $\ga_{1}$ is not minimizing in a larger interval. Although this fact is well-known, let us sketch a proof. Assume by contradiction the existence of $t>L$ such $\ga_1:[0,t]\to\hh^n$ is minimizing. Then the concatenation of
$\ga_{2}:[0,L]\to\hh^n$ and $\ga_{1}:[L,t]\to\hh^n$ is also a a minimizing geodesic. By Pontryagin Maximum Principle, this concatenation is also regular. Since it coincides with $\ga_2$ in the non-trivial interval $[t,L]$, the uniqueness of geodesics implies that
$\ga_{1}=\ga_{2}$ on $[0,L]$, a contradiction.

From \S~2.3, two different geodesics of curvature $\la\neq 0$ 
extending from a given point $p\in\hh^n$, meet again for
$s=2\pi/|\la|$. Hence geodesics of curvature $\la$ are minimizing in
intervals of length $2\pi/|\la|$, but not on larger ones. This property also follows from the second variation of length as indicated by Rumin \cite[p.~327]{MR1267892}. 

Let us show that the Carnot-Carath\'eodory distance is in fact a
smooth function in the Euclidean sense outside a vertical line, see \cite{MR2035027} and \cite{MR1976833} for the $\hh^1$ case

\begin{lemma}
Let $p\in\hh^n$ and let $L_{p}$ be the vertical axis passing through
$p$.  Then the distance function $d_{p}(q):=d(p,q)$ is analytic,
with non-vanishing Euclidean gradient, in $\hh^n\setminus L_{p}$.
\end{lemma}

\begin{proof}
Since left-translations preserve the Carnot-Carathéodory distance and 
the vertical lines, it is enough to prove the result for $p=0$.

Let $U:=\{(v,\la,s)\in\sph^{2n-1}\times\rr\times\rr^+ : |v|=1,
|\la s|<\pi\}$, and define the $C^\infty$ map $E:U\to\rr^{2n+1}$ by
\[
E(v,\la,s):=\ga_{0,v}^\la(s)=(s\,F(\la s)\,v-s\,G(\la s)\,J(v), s^2 
H(\la s)).
\]
It is straightforward to check that $E$ is an injective mapping, and
that its image is $\rr^{2n}\setminus L_{0}$.

We compute the matrix of the differential $dE_{(v,\la,s)}$. Consider
the orthonormal basis of $\rr^{2n}$ given by
$\{J(w_{1}),w_{1},\ldots,J(w_{n-1}),w_{n-1},J(v),v\}$. Then
\[
\{J(w_{1}),w_{1},\ldots,J(w_{n-1}),w_{n-1},J(v)\}
\]
is an orthonormal 
basis of $T_{v}\sph^{2n-1}$. We identify these vectors with ones in
the tangent space to $\sph^{2n-1}\times\rr\times\rr^+$ at $(v,\la,s)$.
We denote by $\ptl_{\la}$ and $\ptl_{s}$ the tangent vectors to
the coordinates $\la$ and $s$. We consider in $\rr^{2n+1}$ the
orthonormal basis
\[
\{J(w_{1}),w_{1},\ldots,J(w_{n-1}),w_{n-1},J(v),v,\ptl/\ptl t\}.
\]
Let $w$ be one of the vectors $w_{i}$, $J(w_{i})$. Then from \eqref{eq:geodesic2} we have
\begin{align*}
dF_{(v,\la,s)}(w)&=(s\,F(\la s)\,w-s\,G(\la s)\,J(w),0),
\\
dF_{(v,\la,s)}(\ptl_{\la})&=(2s^2 F'(\la s)\,v-2s^2 G'(\la s)\,J(v),
2s^3 h'(\la s)),
\\
dF_{(v,\la,s)}(\ptl_{s})&=((F(\la s)+\la s\,F'(\la s))\,v
-(G(\la s)+\la s\,G'(\la s))\,J(v),
\\
&2s\,H(\la s)+2s^3 H'(\la
s)),
\end{align*}
and so we obtain that the determinant of $dF_{(v,\la,s)}$ in the above
basis is given by
\[
\bigg(2s^2\,\frac{G(\la s)}{\la s}\bigg)^{n-1}\,
\det
\begin{pmatrix}
s F(\la s) & -2 s^2 G'(\la s) & -G(\la s)
\\
s G(\la s) & 2s^2 F'(\la s) & F(\la s)
\\
0 & 2 s^3 H'(\la s) & 2s\,H(\la s)
\end{pmatrix},
\]
which is equal to
\[
\bigg(2s^2\,\frac{G(\la s)}{\la s}\bigg)^{n-1}\,
\bigg(\frac{-1+\cos(\la s)+\la s\,\sin(\la s)}{\la^4}\bigg),
\]
and hence to
\[
\bigg(2s^2\,\frac{G(\la s)}{\la s}\bigg)^{n-1}\,
\bigg(\frac{\sin(\la s)\,\big(\la s\,\cos(\la s)-\sin(\la s)\big)}{\la^4}\bigg).
\]
Observe that $G(x)/x=g(x)$ is an analytic function that does not vanish in
the interval $(-2\pi,2\pi)$, and that $\sin(x)/x$ and
$(x\,\cos(x)-\sin(x))/x^3$ are analytic functions that do not vanish
in the interval $(-\pi,\pi)$.  Hence the Jacobian of
$dE_{(v,\la,s)}$ does not vanish, so that $E$ is a local analytic
diffeomorphism.

Since $E$ is a bijective mapping, it is also a global diffeomorphism.  Hence, for any point $p\in\hh^n\setminus L_{0}$ there exists a unique geodesic $\ga_{0,v}^\la$ so that $p=\ga_{o,v}^\la(s)$ and $|\la s|<\pi$.  We conclude that $\ga_{0,v}^\la$ is minimizing and that $s=d(p,0)$.

Composing the local inverse of $E$ with the projection over the coordinate $s$ we conclude that the Carnot-Carath\'eodory distance is an analytic function out of the vertical axis passing through the
origin.
\end{proof}

A simple corollary of the previous result is

\begin{lemma}
Let $W=\{(p,q)\in\hh^n\times\hh^n : \pi(p)\neq\pi(q)\}$. Then
$d:W\to\rr$ is an analytic function.
\end{lemma}

\begin{proof}
We simply consider the map
\begin{align*}
\hh^n\times U&\to\hh^n\times\hh^n
\\
(p,(v,\la,s))&\mapsto (p,\ga_{p,v}^\la(s)),
\end{align*}
which is locally invertible by the previous Lemma. Composing the local
inverse with the projection over $s$ we obtain the desired result.
\end{proof}

The analyticity of the distance function has been recently treated by Haj\l asz and Zimmerman in \cite{MR3417082}.

\subsection{Variations by geodesics}
We shall often use variations by horizontal curves. The existence of such variations is guaranteed by the following result

\begin{lemma}
\label{lem:horcurves}
Let $\ga:[0,a]\to\hh^n$ be a smooth horizontal curve, $I\subset\rr$ an open interval containing the origin, and $\alpha,\beta:I\to\hh^n$ smooth curves such that $\alpha(0)=\ga(0)$, $\beta(0)=\ga(a)$. Then there exist $\eps_0>0$ and a variation $\ga_\eps$, $|\eps|<\eps_0$, of $\ga$ by horizontal curves such that $\ga_\eps(0)=\alpha(\eps)$ and $\ga_\eps(a)=\beta(\eps)$ for $|\eps|<\eps_0$.
\end{lemma}

\begin{proof}
We decompose $\ga=(\ga_1,\ga_2)\in\rr^{2n}\times\rr$. Consider the two-parameter family of curves $\Ga_{\eps,\rho}:[0,a]\to\rr^{2n}$ defined by
\begin{equation}
\label{eq:defgaepsrho}
\Ga_{\eps,\rho}(s)=\ga_1(s)+\eps U(\eps,s)+\rho V(s).
\end{equation}
Here $U(\eps,s)$ is defined by
\[
U(\eps,s)=\bigg(1-\frac{s}{a}\bigg)\,\rho_\alpha(\eps)+\frac{s}{a}\,\rho_\beta(\eps),
\]
where
\begin{equation*}
\alpha_1(\eps)=\alpha_1(0)+\eps\rho_\alpha(\eps),
\quad
\beta_1(\eps)=\beta_1(0)+\eps\rho_\beta(\eps),
\end{equation*}
and $V(s)$ is a vector field along $\ga_1(s)$ vanishing at $0$ and $a$ and such that
\[
\int_0^a\escpr{J(\dot{V}),V}(\xi)\,d\xi\neq 0.
\]
It is enough to take $V(s)=\cos(\tfrac{2\pi}{a}s)\tfrac{\ptl}{\ptl x_1}+\sin(\tfrac{2\pi}{a}s)\tfrac{\ptl}{\ptl y_1}$.

Now we consider the two-parameter family of horizontal curves in $\hh^n$ given by
\begin{equation}
\Lambda_{\eps,\rho}(s):=\big(\Ga_{\eps,\rho}(s),\alpha_2(\eps)+\frac{1}{2}\int_0^s\escpr{J(\dot{\Ga}_{\eps,\rho}),\Ga_{\eps,\rho}}(\xi)\,d\xi\big).
\end{equation}
When $\eps=\rho=0$ we have from \eqref{eq:defgaepsrho} that $\Ga_{0,0}=\ga_1$ and so $\Lambda_{0,0}=\ga$. For $s=0$, we have
\[
\Lambda_{\eps,\rho}(0)=\big(\ga_1(0)+\eps\rho_\alpha(\eps),\alpha_2(\eps)\big)=(\alpha_1(\eps),\alpha_2(\eps))=\alpha(\eps)
\]
for any $\eps$, $\rho$. For $s=a$ we have
\[
(\Lambda_{\eps,\rho})_1(a)=\Ga_{\eps,\rho}(a)=\ga_1(a)+\eps\rho_\beta(\eps)=\beta_1(\eps).
\]
So it remains to prove we can choose $\rho(\eps)$ for $\eps$ small so that $\rho(0)=0$ and
\[
(\Lambda_{\eps,\rho(\eps)})_2(a)=\alpha_2(\eps)+\frac{1}{2}\int_0^a \escpr{J(\dot{\Ga}_{\eps,\rho(\eps)}),\Ga_{\eps,\rho(\eps)}}(\xi)\,d\xi=\beta_2(\eps).
\]
We define a smooth function of two variables $G:I\times\rr\to\rr$ by
\[
G(\eps,\rho):=(\alpha_2(\eps)-\beta_2(\eps))+\frac{1}{2}\int_0^a\escpr{J(\dot{\Ga}_{\eps,\rho}),\Ga_{\eps,\rho}}(\xi)\,d\xi.
\]
This function satisfies
\[
G(0,0)=\ga_1(0)-\ga_2(0)+\frac{1}{2}\int_0^a\escpr{J(\dot{\ga}_1),\ga_1)}(\xi)\,d\xi=0
\]
as $\ga$ is a horizontal curve. Moreover
\[
\frac{\ptl G}{\ptl\rho}(0,0)=\frac{1}{2}\int_0^a\escpr{J(\dot{V}),V}(\xi)\,d\xi\neq 0
\]
by the choice of $V$. By the Implicit Function Theorem, there exist $\eps_0>0$ and a function $\rho:(-\eps_0,\eps_0)\to\rr$ such that $\rho(0)=0$ and $G(\eps,\rho(\eps))=0$ for all $\eps\in (-\eps_0,\eps_0)$. This implies that $\ga_\eps=\Lambda_{\eps,\rho(\eps)}:[0,a]\to\hh^n$ is a variation of $\ga$ by smooth horizontal curves joining $\alpha(\eps)$ and $\beta(\eps)$.
\end{proof}

In the next result we compute the derivative of length when we deform a sub-Riemannian geodesic by horizontal curves

\begin{lemma}
\label{lem:1stvarlen}
Let $\ga:[0,a]\to\hh^n$ be a geodesic of curvature $\la$, and $\{\ga_\eps\}_{\eps}$ a variation of $\ga$ by horizontal curves. Let $U(s):=(\ptl\ga_{\eps}/\ptl\eps)(s)$. Then
\begin{equation}
\label{eq:1stvarlen}
\frac{d}{d\eps}\bigg|_{\eps=0} L(\ga_{\eps})=\escpr{U,\dot{\ga}+\tfrac{\la}{2}\,T_{\ga}}\big|_{0}^{a}.
\end{equation}
\end{lemma}

\begin{proof}
We have
\begin{align*}
\frac{d}{d\eps}\bigg|_{\eps=0} L(\ga_{\eps})&=\int_{0}^{a} \escpr{\nabla_{U}\dot{\ga},\dot{\ga}}
=\int_{0}^{a} \escpr{\nabla_{\dga}U+\tor(U,\dga),\dga}
\\
&=\int_{0}^{a}
\big(\dot{\ga}\,\escpr{U,\dot{\ga}}-\escpr{U,\nabla_{\dot{\ga}}\dot{\ga}}\big),
\end{align*}
since $[U,\dga]=0$ and $\tor(U,\dga)$ is a vertical vector. From equations \eqref{eq:geodesic} and \eqref{eq:admissible} we get
\[
-\escpr{U,\nabla_{\dga}\dga}=\escpr{U,\la J(\dga)}=\tfrac{\la}{2}\,\dga\,\escpr{U,T},
\]
from which \eqref{eq:1stvarlen} follows.
\end{proof}

\subsection{A second fundamental form for $C^2$ hypersurfaces}
\label{sec:hor2nd}
In this subsection, assume that $S\subset\hh^n$ is an embedded hypersurface of class $C^2$. Let $S_0$ be the \emph{singular set} of points $q\in S$ where $T_qS$ coincides with the horizontal distribution. Let $N$ be a unit normal to $S$ and $\nuh$ the horizontal unit normal, defined by
\begin{equation}
\label{eq:defnuh}
\nuh:=N_h/|N_h|.
\end{equation}
The characteristic vector field $Z$ is defined by
\begin{equation}
\label{eq:defZ}
Z:=J(\nuh).
\end{equation}
It is a horizontal vector field defined in $S\setminus S_0$ and tangent to $S$. The vector field $\escpr{N,T}\,\nuh-\mnh \,T$, tangent to $S$ and not horizontal, will be often considered along this paper.

If $q\in S\setminus S_0$ and $u\in T_qS\cap\hhh_q$, we define the \emph{horizontal second fundamental form} of $S$ by 
\begin{equation}
\label{eq:sigma}
A(u)=-\nabla_u\nuh-\ntnh\,J(u)_{ht},
\end{equation}
where by $U_{ht}$ we denote the tangent horizontal projection onto $S$, defined by
\[
U_{ht}=U-\escpr{U,T}T-\escpr{U,\nuh}\nuh
\]
for any vector field $U$. The operator $A:T_qS\cap\hhh_q\to T_qS\cap\hhh_q$ defined by \eqref{eq:sigma} was introduced in \cite{MR2898770} and studied in \cite{cchy} and \cite{MR3504195} . 

Given a $C^1$ function $f:S\to\rr$, we define its \emph{horizontal gradient on} $S$ as the~tangent and horizontal vector field $\nabla_S^h f$ in $S\setminus S_0$ satisfying $\escpr{(\nabla_S^hf)_q,u}=u(f)$ for any $q\in S\setminus S_0$ and $u\in T_qS\cap\hhh_q$. The following properties are known

\begin{proposition}
\label{prop:sigmaproperties}
Let $S\subset \hh^n$ be an embedded $C^2$ hypersurface with horizontal unit normal $\nuh$ and horizontal second fundamental form $A$. Then we have
\begin{enumerate}
\item $\escpr{A(u),v}=\escpr{u,A(v)}$, $u,v\in T_qS\cap\hhh_q$.
\item For $E=\escpr{N,T}\nuh-|N_h|T$, we have on $S\setminus S_0$:
\begin{equation}
\label{eq:nablaEnuh}
-|N_h|^{-1}\nabla_E\nuh=\nabla_S^h\bigg(\frac{\escpr{N,T}}{|N_h|}\bigg)+2\,\bigg(\frac{\escpr{N,T}}{|N_h|}\bigg)^2J(\nuh).
\end{equation}
\item $Z=J(\nuh)$ is an eigenvector of $A$ if and only if $[U,Z]$ is tangent and horizontal for any vector field $U$ in $S\setminus S_0$ tangent, horizontal and orthogonal to $Z$.
\end{enumerate}
\end{proposition}

\begin{proof}
To prove the symmetry of $A$ take $U, V\in TS\cap\hhh$. Since $\escpr{[U,V],N}=0$ we have
\[
\escpr{\nabla_UV-\nabla_VU-\tor(U,V),N}=0.
\]
As $\nabla_UV,\nabla_VU$ are horizontal and $\tor(U,V)=2\escpr{J(U),V}\,T$ is vertical, we get
\[
-\escpr{V,\mnh\nabla_U\nuh+\escpr{N,T}J(U)}+\escpr{U,\mnh\nabla_V\nuh+\escpr{N,T}J(V)}=0,
\]
that implies the symmetry of $A$.

Let us now prove \eqref{eq:nablaEnuh}. Let $Z=J(\nuh)$. We observe first that $\nabla_E\nuh$ is orthogonal to $\nuh$ and $T$. Let $U$ be any horizontal and tangent vector field on $S\setminus S_0$. Since $[E,U]$ is tangent we have
\[
\escpr{\nabla_EU-\nabla_UE-\tor(E,U),N}=0.
\]
Decomposing $N=\mnh\nuh+\escpr{N,T}T$ we have
\[
\mnh\escpr{\nabla_EU,\nuh}-\escpr{\nabla_UE,N}-2\escpr{N,T}\escpr{J(E),U}=0.
\]
As $U$ is horizontal, $\escpr{\nabla_EU,\nuh}=-\escpr{U,\nabla_E\nuh}$. From the definition of $E$ and $Z$ we obtain $\escpr{N,T}\escpr{J(E),U}=\escpr{N,T}^2\escpr{Z,U}$. Finally
\begin{align*}
\escpr{\nabla_UE,N}=-\escpr{E,\nabla_UN}&=-\escpr{N,T}U(\mnh)+\mnh U(\escpr{N,T})
\\
&=\mnh^2U\bigg(\ntnh\bigg).
\end{align*}
So we obtain
\[
\escpr{U,-\mnh\nabla_E\nuh-\mnh^2U\bigg(\ntnh\bigg)-2\escpr{N,T}^2Z}=0,
\]
that implies \eqref{eq:nablaEnuh}.
\end{proof}

The symmetry of the horizontal second fundamental form proved in Proposition~\ref{prop:sigmaproperties}(1) implies the existence, at every point $q\in S\setminus S_0$, of an orthonormal basis $e_1,\ldots,e_{2n-1}$ of $T_qS\cap\hhh_q$ and of real numbers $\kappa_1,\ldots,\kappa_{2n-1}$ such that
\[
A(e_i)=\kappa_ie_i,\quad i=1,\ldots, (2n-1).
\]
The \emph{mean curvature} $H$ of $S$, defined on $S\setminus S_0$ as the trace of the operator $A$ on $T_qS\cap\hhh_q$, is thus given by
\[
H=\kappa_1+\cdots +\kappa_{2n-1}.
\]
The mean curvature plays a prominent role in some geometric variational problems in sub-Riemannian geometry because of its relation to the first variation of the sub-Riemannian perimeter functional in contact sub-Riemannian manifolds, see e.g. \cite{MR2435652}, \cite{MR3044134}, and \cite{MR2898770}.

The \emph{norm} of the horizontal second fundamental form is the function $|\sg|^2$ defined on $S\setminus S_0$ by the formula
\begin{equation}
\label{eq:sigma2}
|\sg|^2=\sum_{i=1}^{2n-1} \kappa_i^2.
\end{equation}
Taking an orthonormal basis $e_1,\ldots,e_{2n-1}$ of the tangent horizontal space $TS\cap \hhh$, the trace $\sum_{i=1}^{2n-1}|\nabla_{e_i}\nuh|^2$ does not depend on the basis. The function $|\sg|^2$ can be computed this way since \eqref{eq:sigma2} is the trace corresponding to an orthonormal basis of principal directions.

Following Cheng et al.~\cite{cchy}, we define an \emph{umbilic} hypersurface in the sub-Rie\-mannian Heisenberg group $\hh^n$ as one for which $Z$ is a principal direction, and the remaining ones have equal principal curvatures. More precisely, there exists and orthonormal basis $e_1=Z,e_2,\ldots,e_{2n-1}$, and scalars $\rho,\mu$ such that $A(e_1)=\rho e_1$ and $A(e_i)=\mu e_i$ for $i\ge 2$. Hypersurfaces of revolution are umbilic by Proposition~3.1 in \cite{cchy}.

The following result can be deduced from Proposition~4.2 in \cite{cchy}

\begin{proposition}
\label{prop:cchy}
Let $S\subset\hh^n$ be an umbilic hypersurface with $A(Z)=\rho Z$ and $A(V)=\mu V$ for all tangent horizontal vectors $V$ orthogonal to $Z$. Then we have
\begin{equation}
V(\mu)=V(\rho)=V\bigg(\ntnh\bigg)=0,\quad Z(\mu)=(\rho-2\mu)\,\ntnh,
\end{equation}
and
\begin{equation}
Z\bigg(\ntnh\bigg)+\bigg(\ntnh\bigg)^2=\mu(\mu-\rho).
\end{equation}
and
\begin{equation}
\nabla_S^h\bigg(\ntnh\bigg)=Z\bigg(\ntnh\bigg)\,Z.
\end{equation}
\end{proposition}

\section{The distance function to a closed set}
\label{sec:distance}

Let $E\subset\hh^n$ be a closed set with boundary $S=\ptl E$. We define the distance to $E$ by
\begin{equation}
\label{eq:da}
\de(x):=\inf\{d(x,q) : q\in E\},
\end{equation}
where $d$ is the Carnot-Carathéodory distance in $\hh^n$. The function $\de$ is lipschitz with respect to $d$ as it satisfies $|\de(x)-\de(y)|\le d(x,y)$.  By Pansu-Rademacher's Theorem \cite{MR979599}, the function $\de$ is Pansu-differentiable almost everywhere (see also Calderón's proof of this result in \cite[Thm.~6.13]{MR2312336}). For $r>0$ we define the open tubular neighborhood of $E$ of radius $r>0$ by
\[
E_{r}:=\{p\in\hh^n : \de(p)<r\}.
\]

For $q\in E$, we define the set $\ttan(E,q)$ of tangent vectors to $E$ at $q$ as the subset of $T_{q}\hh^n$ composed of the zero vector and the limits, in the tangent bundle, of the tangent vectors of $C^1$ curves starting from $q$ and contained in $E$. The set $\ttan(E,q)$ is closed and positively homogeneous and will be called the \emph{tangent cone} of $E$ at $q$. Obviously, if $E\subset F$ then $\ttan(E,q)\subset\ttan(F,q)$. The \emph{horizontal tangent cone} $\ttanh(E,q)$ is defined as $\ttan(E,q)\cap\mathcal{H}_{q}$. The \emph{normal cone} $\nor(E,q)$ is defined as the set of  vectors $u\in T_q\hh^n$ such that $\escpr{u,v}\le 0$ for all $v\in\ttan(E,q)$. The set $\nor(E,q)$ is a closed convex cone of $T_{q}\hh^n$. The \emph{horizontal normal cone} $\norh(E,q)$ is defined as the set
\[
\norh(E,q)=\{v\in \hhh_q: \escpr{v,u}\le 0 \text{ for all }u\in\ttanh(E,q)\}.
\]
In general $\norh(E,q)\neq\nor(E,q)\cap\hhh_q$. Observe that if $\hhh_q\subset\ttan(E,q)$ then $\norh(E,q)=\{0\}$.

We shall say that $q\in E$ is a \emph{singular} point if the tangent cone $\ttan(E,q)$ is contained in one of the half-spaces in $T_q\hh^n$ determined by the hyperplane $\hhh_q$. We shall say that a point $q\in E$ is \emph{regular} if it is not singular. The set of singular points of $E$ will be denoted by $E_0$. Observe that interior points of $E$ are regular since $\ttan(E,q)=T_q\hh^n$ when $q\in\intt(E)$. The set of singular points of $E$ at the boundary $S$ of $E$ will be denoted by $S_0$.

For sufficiently regular boundaries we have the following result

\begin{lemma}
\label{lem:cones}
Let $E\subset\hh^n$ be the closure of an open set with boundary $S=\ptl E$. Let $q\in S$.
\begin{enum}
\item Assume that $S$ is of class $C^1$ in an open neighborhood of $q$, and let $N_q$ be the outer unit normal to $S$ at $q$. Then $\ttan(E,q)=\{u\in T_q\hh^n:\escpr{u,N_q}\le 0\}$ and $\nor(E,q)=\{\rho N_q: \rho\ge 0\}$.
\item Assume that $S$ is of class $C_\hh^1$ in an open neighborhood of $q$, and let $\nu_q$ be the outer horizontal unit normal of $S$ at $q$, then $\ttan_H(E,q)=\{u\in\hhh_q : \escpr{u,\nu_q}\le 0\}$ and $\norh(E,q)=\{\rho \nu_q:\rho\ge 0\}$.
\end{enum}
\end{lemma}

\begin{proof}
We shall give the proof of (ii) since the one of (i) is similar.

If $S$ is of class $C_\hh^1$ near $q$, then there exists an open ball $B(q,r)$ and a function $f\in C_\hh^1(B(q,r))$ such that $E\cap B(q,r)=f^{-1}((-\infty,0])$, $\ptl E\cap B(q,r)=f^{-1}(0)$, and $\nu_q=(\nabla_\hh f)_q/|(\nabla_\hh f)_q|$. Take $u\in \hhh_q$ such that $\escpr{u,\nu_q}<0$ and a horizontal curve $\alpha:[0,1]\to\hh^n$ of class $C^1$ satisfying $\alpha'(0)=u$. Since $\escpr{u,\nu_q}=|(\nabla_\hh f)_q|^{-1}\escpr{u,(\nabla_\hh f)_q}<0$ we get $u(f)=\tfrac{d}{ds}|_{s=0} (f\circ\alpha)(s)<0$. Hence there exists $\eps>0$ such that $\alpha([0,\eps])\subset E$. This implies that $u\in \ttanh(E,q)$ and so the inclusion $\{u\in\hhh_q:\escpr{u,\nu_q}<0\}\subset\ttanh(E,q)$ holds. Since $\ttanh(E,q)$ is closed we obtain $\{u\in\hhh_q: \escpr{u,\nu_q}\le 0\}\subset\ttanh(E,q)$. For the opposite inclusion take $u\in\ttanh(E,q)$. The vector $u$ is the limit of a sequence $u_i$ of tangent vectors to $C^1$ curves contained in $E$ for which inequality $\escpr{u_i,\nu_q}\le 0$ holds. Taking limits when $i\to\infty$ we obtain $\escpr{u,\nu_q}\le 0$.

Finally equality $\nor(E,q)=\{\rho \nu_q:\rho\ge 0\}$ follows trivially. 
\end{proof}

\begin{remark}
Lemma~\ref{lem:cones} implies that the set of singular points of a closed set $E$ with $C^1$ boundary $S$ is the union of the interior of $E$ and the points in $S$ with $T_qS=\hhh_q$.
\end{remark}

Following Federer's terminology \cite{MR0110078}, we define $\unp(E)$ as the set of points $p\in\hh^n$ for which there is a unique point of $E$ nearest to $p$. The map $\xi_{E}:\unp(E)\to E$ associates with $p\in\unp(E)$ the unique $q\in E$ such that $\de(p)=d(p,q)$. Trivially $E\subset\unp(E)$ and $\xi_E(q)=q$ for all $q\in E$.

For $q\in E$ we define $\reach(E,q)$ as the supremum of $r>0$ for which $B(q,r)\subset\unp(E)$.  If $K\subset E$ then $\reach(E,K)$ is defined as the  infimum of $\reach(E,q)$ for $q\in K$. The reach of a set $E$ is defined by
\[
\reach(E):=\inf\{\reach(E,q) : q\in E\}.
\]
Observe that the function $q\mapsto \reach(E,q)$ is continuous.

Let us now prove that the metric projection $\xi_E$ is a continuous function in $\unp(E)$.

\begin{proposition}
\label{prop:xicont}
Let $E\subset\hh^n$ be a closed set. Then the function $\xi_E:\unp(E)\to E$ is continuous.
\end{proposition}

\begin{proof}
Let $\xi=\xi_E$.  Consider a sequence $\{p_{i}\}_{i\in\nn}\subset\unp(A)$ converging to $p\in\unp(A)$. Let us prove that $\xi(p_{i})\to \xi(p)$ by contradiction: otherwise, passing to a subsequence, we may assume that there exists some $\eps>0$ such that
\begin{equation}
\label{eq:eps}
d(\xi(p_{i}),\xi(p))\ge\eps.
\end{equation}
Observe that the quantity
\[
d(\xi(p_{i}),p)\le d(\xi(p_{i}),p_{i})+d(p_{i},p)=\delta_{E}(p_{i})+
d(p_{i},p),
\]
is bounded. Hence $\xi(p_{i})$ is bounded and we may assume, passing again to a subsequence, that $\xi(p_i)$ converges to some point $q\in E$. By the continuity of $\delta_{E}$ and $d$ we have
\[
\delta_{E}(p)=\lim_{i\to\infty} \delta_{E}(p_{i})=\lim_{i\to\infty}
d(\xi(p_{i}),p_{i})=d(q,p).
\]
Since $p\in\unp(E)$ we have $\xi(p)=q$ and, since $q=\lim_{i\to\infty}\xi(p_{i})$, we get a contradiction to \eqref{eq:eps} that proves the continuity of $\xi_E$.
\end{proof}

Let $p\in\unp(E)\setminus E$.  Then there is either just one minimizing geodesic connecting $p$ and $\xi_E(p)$ of curvature $\la(p)$ (in case $p$ and $\xi_E(p)$ do not lie in the same vertical line), or there are at least two minimizing geodesics connecting $p$ and $\xi_E(p)$, and all geodesics joining connecting both points have the same geodesic curvature $\la(p)$ (in case $p$ and $\xi_E(p)$ lie in the same vertical line). 

So even if the minimizing geodesic connecting $p$ and $\xi_E(p)$ is not unique, the quantity $\la(p)$ is well defined. Let us see that it is a continuous function

\begin{proposition}
\label{prop:lacont}
Let $E\subset\hh^n$ be a closed set.  Then $\la:\unp(E)\setminus E\to\rr$ is a continuous function.
\end{proposition}

\begin{proof}
Consider a sequence $\{p_{i}\}_{i\in\nn}\subset\unp(E)$ converging to a point $p\in\unp(E)\setminus E$. Let us prove the continuity that $\lim_{i\to\infty}\la(p_i)=\la(p)$ by contradiction: passing eventually to a subsequence, we assume that there exists some $\eps>0$ such that
\begin{equation}
\label{eq:eps2}
|\la(p_{i})-\la(p)|\ge\eps.
\end{equation}

For every $i$, choose a minimizing geodesic connecting $\xi(p_i)$ and $p_i$ with initial velocity $v_i$. Passing again to a subsequence if necessary we may assume that $v_i$ converges to some unit vector $v_0$.

By the minimality of geodesics, $|\la(p_{i})|\delta_{E}(p_{i})\le 2\pi$ and, since $\delta_{E}(p_{i})\to\delta_{E}(p)>0$, we have that $|\la(p_{i})|$ is bounded. Passing again to a subsequence, we may assume that $\la(p_{i})$ converges to some $\la_0\in\rr$.

Since
$\xi(p_{i})\to\xi(p)$ and $\delta_{E}$ is continuous we have that
\[
\delta_{E}(p)=\lim_{i\to\infty}\delta_{E}(p_{i})=
\lim_{i\to\infty}\ga_{\xi(p_{i}),v(p_{i})}^{\la(p_{i})}(\delta_{E}(p_{i}))=
\ga_{\xi(p),v_0}^{\la_0}(\delta_{E}(p)).
\]
By the uniqueness of curvatures of minimizing geodesics, we
obtain that $\la(p)=\la_{0}$, which contradicts \eqref{eq:eps2}.
\end{proof}

\begin{remark}
\label{rem:regunique}
Let $p\not\in E$ and $q\in E$ such that $\delta_E(p)=d(p,q)$. Assume there is not a unique length-minimizing geodesic connecting $p$ and $q$. Then $p$ and $q$ lie in the same vertical line. The set
\[
S=\bigcup_{v\in\hhh_p} \ga_{p,v}^{\la(p)}([0,2\pi/\la(p)])
\]
is a $C^2$ sphere (Pansu's sphere) and $S\setminus\{q\}$ is contained in $\hh^n\setminus E$. Hence the tangent cone $\ttan(E,q)$ is contained in one of the half-spaces determined by the hyperplane $\hhh_q$. This implies that $q$ is a singular point of $E$.

Hence, if $q$ is regular then the length-minimizing geodesic connecting $p$ and $q$ is unique.
\end{remark}

Let $p\in\unp(E)\setminus E$ such that $\xi(p)$ is a regular point. We shall denote by $v(p)$ the initial velocity of the unique minimizing geodesic connecting $\xi(p)$ to $p$.

\begin{proposition}
\label{prop:vcont}
Let $E\subset\hh^n$ be a closed set. Consider a sequence $\{p_{i}\}_{i\in\nn}\subset\unp(E)\setminus E$ converging to a point $p\in\unp(E)\setminus E$.  Assume that $\xi(p)$ is a regular point. Let $\{v_{i}\}_{i\in\nn}$ be a sequence of initial tangent vectors to minimizing geodesics connecting $\xi(p_{i})$ to $p_{i}$. Then $\lim_{i\to\infty}v_{i}=v(p)$.

In particular, the function $q\mapsto v(q)$, assigning to $q\in\xi^{-1}(S\setminus S_{0})$ the initial tangent vector to the unique geodesic connecting $\xi(q)$ and $q$, is continuous in $\xi^{-1}(S\setminus S_{0})$.
\end{proposition}

\begin{proof}
In case the sequence $v_{i}$ does not converge to $v(p)$, we may extract a convergent subsequence to some vector $v_{0}\neq v(p)$. Since $\xi(p_{i})\to\xi(p)$, $\la(p_{i})\to\la(p)$ and $\delta_{E}(p_{i})\to \delta_{E}(p)$, passing to a subsequence we may assume that
\[
\lim_{i\to\infty}\ga_{\xi(p_{i}),v_{i}}^{\la(p_{i})}(\delta_{A}(p_{i}))
=\ga_{\xi(p),v_{0}}^{\la(p)}(\delta_{A}(p)).
\]
By Remark~\ref{rem:regunique}, the regularity of $\xi(p)$ implies the existence of a unique minimizing geodesic connecting $\xi(p)$ and $p$.  By the above formula this geodesic should be $\ga_{\xi(p),v_{0}}^{\la(p)}$. Hence we would have $v_{0}=v(p)$, yielding a contradiction.
\end{proof}

We conclude this section by considering the regularity of the boundaries of tubular neighborhoods of a set $E$ in $\unp(E)$. We shall need first the following Lemma

\begin{lemma}
\label{lem:federer}
Let $E\subset\hh^n$ be a closed set. Let $\delta=\delta_E$, $\xi=\xi_E$.
\begin{enum}
\item Let $f$ be a lipschitz function on an open subset $\Omega\subset\hh^n$. Let $X$ be a continuous vector field on $\Om$ such that $(\nabla_\hh f)_q=X_q$ whenever $f$ is $\hh$-differentiable at $q$. Then $(\nabla_\hh f)_q=X_q$ for all $q\in\Om$.
\item If $p\not\in E$ and $\delta$ is $\hh$-differentiable at $p$, then
\begin{equation}
\label{eq:nabladeltaE}
(\nabla_\hh\delta)_p=\ga'(\delta(p)),
\end{equation}
where $\ga:[0,\delta(p)]\to\hh^n$ is any minimizing geodesic connecting $\ptl E$ and $p$.
\item Let $A$ be the interior of the set $(\unp(E)\setminus E)\cap\xi^{-1}(\ptl E\setminus E_0)$. Then the horizontal gradient $\nabla_\hh\delta$ is continuous in $A$ and so $\delta\in C_\hh^1(A)$.
\end{enum}
\end{lemma}

\begin{proof}
The proof of (i) was given in Lemma~6.1 in \cite{MR2299576}. It is inspired by Lemma~4.7 in Federer \cite{MR0110078}.

To prove (ii) consider a point $p\not\in E$ where $\delta$ is $\hh$-differentiable and let $q\in S$ be a point in $E$ at minimum distance from $p$. Take a minimizing  geodesic $\ga:[0,\delta(p)]\to\hh^n$ connecting $q$ and $p$. Let $\alpha:(-\eps,\eps)\to\hh^n$ be a $C^1$ horizontal curve satisfying $\alpha(0)=p$ and $\alpha'(0)=u$, and define $f(s):=d(\alpha(s),q)$. Then $\delta_E(\alpha(s))\le f(s)$ and $\delta_E(\alpha(0))=f(0)$. Since $\delta$ is assumed to be $\hh$-differentiable at $p$ we have $\frac{d}{ds}|_{s=0}\delta(\alpha(s))=f'(0)$. The derivative $f'(0)$ can be computed using \eqref{eq:1stvarlen} to obtain

\begin{equation}
\label{eq:gradient}
\escpr{(\nabla_\hh\delta)_p,u}=\lim_{s\to 0}\frac{\delta_E(\alpha(s))-\delta_E(\alpha(0))}{s}= f'(0)=\escpr{u,\ga'(\delta(p))}.
\end{equation}

Now we prove (iii). By Pansu-Rademacher Theorem, the lipschitz function $\delta$ is $\hh$-differentiable almost everywhere. For any point $p\in A$, its metric projection $\xi(p)$ is a regular point. If  $\delta$ is $\hh$-differentiable at $p$, (ii) implies that its horizontal gradient coincides with the vector field
\[
p\in U\mapsto (\gamma_{\xi(p),v(p)}^{\la(p)})'(\delta(p),
\]
that it is continuous because of Propositions~\ref{prop:xicont}, \ref{prop:lacont} and \ref{prop:vcont}. We conclude from (i) that $(\nabla_\hh\delta)_q$ exists for any $q\in A$ and it is continuous.
\end{proof}

\begin{remark}
Formula \eqref{eq:nabladeltaE} implies that, at a point $p$ of differentiability of $\delta$, all minimizing geodesics connecting $E$ with $p$ have the same tangent vector at $p$. This condition guarantees uniqueness of geodesics in Riemannian geometry since they only depend on the initial position and velocity. The dependence of sub-Riemannian geodesics in $\hh^n$ on curvature prevents the same conclusion for the Carnot-Carathéodory distance.
\end{remark}

\begin{remark}
If $p$ projects to two different points in $E$, the minimizing geodesics joining $E$ to $p$ are not minimizing beyond $p$ because of the regularity of geodesics.
\end{remark}

Let us now prove that the boundaries of tubular neighbourhoods of $\hh$-regular hypersurfaces are also regular.

\begin{theorem}
Let $E\subset\hh^n$ be a closed set with boundary $S$. For $r>0$ let $S_r=\ptl E_r$. Consider the set $A=\intt(\unp(E)\setminus E)$, and take an open set $U\subset A$. If $S\cap\xi(U)$ is an $\hh$-regular hypersurface, then $S_r\cap U$ is also $\hh$-regular.
\end{theorem}

\begin{proof}
Observe first that $S_0$ is empty for an $\hh$-regular hypersurface since $\norh(E,q)$ is one-dimensional for any $q\in S$ by Lemma~\ref{lem:cones}(ii). Lemma~\ref{lem:federer}(iii) implies that the function $\delta$ is in $C_\hh^1(\intt(\unp(E)\setminus E))$. From Lemma~\ref{lem:federer}(ii) we get $(\nabla_\hh\delta)_q\neq 0$ for every $q\in\intt(\unp(E)\setminus E)$. Hence $S_r\cap U$ is $\hh$-regular since it is the level set of a $C_\hh^1$ function with non-vanishing horizontal gradient.
\end{proof}

\medskip

Now we start the study of the regularity of the distance function to a closed set.
The following Lemma considers geodesics minimizing the distance to a closed set and would be essential for what follows

\begin{lemma}
\label{lem:mingeo}
Let $E\subset\hh^n$ be a closed set. Take $p\not\in E$ and $q\in E$ such that $\delta(p)=d(p,q)$. Let $\ga:[0,\delta(p)]\to\hh^n$ be a length-minimizing geodesic of curvature $\la$ joining $q$ and $p$. Then
\begin{enum}
\item $\dot{\ga}(0)\in\norh(E,q)$.
\item The curvature $\la$ of the geodesic $\ga$ lies in the interval
\begin{equation}
\label{eq:interval}
\bigg[\sup_{\substack{v\in\ttan(E,q) \\ \escpr{v,T_{q}}< 0}}\frac{-2\,\escpr{v,\dot{\ga}(0)}}{\escpr{v,T_{q}}},
\inf_{\substack{v\in\ttan(E,q) \\ \escpr{v,T_{q}}> 0}}\frac{-2\,\escpr{v,\dot{\ga}(0)}}{\escpr{v,T_{q}}}\bigg],
\end{equation}
where the left quantity is replaced by $-\infty$ if $\escpr{v,T_q}\ge 0$ for all $v\in\ttan(E,q)$, and the right quantity by $+\infty$ when $\escpr{v,T_q}\le 0$ for all $v\in\ttan(E,q)$.
\end{enum}
\end{lemma}

\begin{proof}
Let $\alpha:[0,\eps_{0})\to E$ be a smooth curve with $\alpha(0)=q$, $\alpha'(0)=v$.  Using Lemma~\ref{lem:horcurves} we construct a variation of $\ga$ by smooth horizontal curves $\ga_{\eps}:[0,d(p,q)]\to\hh^n$ joining $\alpha(\eps)$ and $p$.  Let $U(s):=(\ptl\ga_{\eps}/\ptl\eps)(s)$. Consider the function $f(\eps):=L(\ga_{\eps})$.  As $U(0)=v$ and $U(d(p,q))=0$, equation \eqref{eq:1stvarlen} implies
\[
f'(0)=-\escpr{v,\dga(0)+\tfrac{\la}{2}\,T_{q}}.
\]
As $f(\eps)\ge\de(p)$ and $f(0)=\de(p)$ for $\eps>0$ we have $f'(0)\ge 0$. 
Hence we obtain
\begin{equation}
\label{eq:ineqf'}
0\le -\escpr{v,\dot{\ga}(0)+\tfrac{\la}{2}\,T_{q}}.
\end{equation}
By approximation inequality \eqref{eq:ineqf'} holds for any $v\in\ttan(E,q)$. In particular $\escpr{v,\dga(0)}\le 0$ for all $v\in\ttanh(E,q)$, which implies (i). From \eqref{eq:ineqf'} we also have
\[
\la\,\escpr{v,T_{q}}\le -2\,\escpr{v,\dga(0)},
\]
which implies (ii).
\end{proof}

In case the interval defined in \eqref{eq:interval} is empty, there are no minimizing geodesics joining a point outside $E$ with $q$. When $S$ is a Euclidean $C^1$ hypersurface, we can prove that the normal horizontal cone is generated by the outer horizontal unit normal $\nu$ and the interval in \eqref{eq:interval} is a single point.

\begin{theorem}
\label{thm:exp}
Let $E\subset\hh^n$ be a closed subset with $C^{1}$ boundary $S$.  Let $N$ be the outer unit normal to $S$ and $\nu$ the corresponding horizontal unit normal. Take $p\not\in E$ and $q\in S$ such that $\delta(p)=d(p,q)$, and consider a minimizing geodesic $\ga:[0,\delta(p)]\to\hh^n$ of curvature $\la$ connecting $q$ and $p$.  Then
\begin{enum}
\item $q$ is a regular point of $S$, 
\item $\dga(0)=\nu_{q}$, and 
\item the curvature of $\ga$ is given by
\begin{equation}
\label{eq:lamin}
\la=\frac{2\escpr{N,T}}{|N_{h}|}(q).
\end{equation}
\end{enum}
Moreover, in case $N$ is a Euclidean lipschitz vector field, the function $\la$ is locally lipschitz in $S\setminus S_{0}$.
\end{theorem}

\begin{proof}
When $S$ is a Euclidean $C^{1}$ hypersurface and $q\in S$, the horizontal normal cone $\norh(E,q)$ is either $\{0\}$ when $q$ is a singular point, or $\{\mu\,\nu_{q}:\mu\ge 0\}$ when $q$ is regular. Lemma~\ref{lem:mingeo} then implies that $q$ must be a regular point and that $\dga(0)=\nu_{q}$. This proves (i) and (ii).

If $q\in S$ then $\ttan(E,q)=\{v\in T_{q}\hh^{n}: \escpr{v,N_{q}}\le 0\}$. Since $N=\mnh\,\nu+\escpr{N,T}\,T$ we obtain
\[
-\escpr{v,\nu_{q}}\,\mnh(q)\ge \escpr{v,T_{q}}\,\escpr{N_{q},t_{q}}
\]
for all $v\in\ttan(A,q)$. In case $\escpr{v,T_{q}}>0$, the above inequality implies
\begin{equation}
\label{eq:infge0}
\inf_{\substack{v\in\ttan(A,q) \\ \escpr{v,T_{q}}< 0}} \frac{-2\,\escpr{v,\nu_{q}}}{\escpr{v,T_{q}}}\ge \frac{2\escpr{N,T}}{\mnh}(q).
\end{equation}
Taking $v:=-\escpr{N_{q},T_{q}}\,\nu_{q}+\mnh(q)\,T_{q}$, we have $v\in T_{q}S\subset \ttan(A,q)$, $\escpr{v,T_{q}}>0$ and $-\escpr{v,\nu_{q}}\,\mnh(q)=\escpr{v,T_{q}}\,\escpr{N_{q},t_{q}}$, so that equality holds in \eqref{eq:infge0}. The case $\escpr{v,T_{q}}<0$ is handled in a similar way. Hence the interval \eqref{eq:interval} is reduced to the point $(2\,\escpr{N,T}/\mnh)(q)$. This proves (iii).
\end{proof}

\begin{remark}
Assume that $S$ is locally defined as a level set of a function $g:\Om\to\rr$ of class $C^1$ on an open set $\Om\subset\hh^1$. Then we may assume that $S\cap \Om=\{x : g(x)=0\}$ and a horizontal unit normal and a unit normal are given by
\[
\nuh=\frac{\nabla_hg}{|\nabla_hg|},\quad N=\frac{\nabla g}{|\nabla g|},
\]
where $\nabla_h$ is the orthogonal projection of the gradient $\nabla$ to the horizontal distribution. So we have
\[
\frac{2\escpr{N,T}}{|N_{h}|}=\frac{2\,T(g)/|\nabla g|}{|\nabla_hg|/|\nabla g|}=-\frac{[X,Y](g)}{|\nabla_{\hh}g|},
\]
which only depends on the horizontal derivatives of $g$.

The importance of the function $-[X,Y](g)/|\nabla_hg|$ for surfaces in $\hh^1$ was recognized by Arcozzi and Ferrari, who called it the imaginary curvature of $S$. The interested reader is referred to the detailed discussion in the introduction of \cite{MR2386836}. A remarkable property of the function $\la$ is its differentiability along tangent horizontal directions in minimal surfaces of class $C^1$ in $\hh^1$, see Lemma~4.4(3) in \cite{MR3406514}.
\end{remark}

Theorem~\ref{thm:exp} implies that the reach of a point approaches $0$ when we approach a singular point since the curvature of a minimizing geodesic approaches $\infty$ by \eqref{eq:lamin}. Hence the geodesic is minimizing in smaller and smaller intervals near the singular point.

\begin{corollary}
\label{cor:reach0}
Let $E\subset\hh^n$ be a closed set with $C^{1}$ boundary $S$. Then $\reach(E,q)$ approaches $0$ when $q\in S$ approaches the singular set $S_{0}$.
\end{corollary}

\begin{proof}
This follows easily since $\text{reach}(E,q)$ is no larger that the length of a minimizing geodesic leaving $q$, that is smaller than or equal to $2\pi/\la(q)=\pi (\mnh/\escpr{N,T})(q)$.
\end{proof}

For sets with local $C_\hh^1$ boundary we have

\begin{theorem}
\label{thm:ch1exp}
Let $E\subset\hh^n$ be a closed set with boundary $S$. Take $p\not\in E$ and $q\in S$ such that $\delta(p)=d(p,q)$, and consider a minimizing geodesic $\ga:[0,\delta(p)]\to\hh^n$ connecting $q$ and $p$. Assume that $S$ is $\hh$-regular near $q$. Then $\dga(0)=\nu_q$, where $\nu$ is the outer horizontal unit normal to $S$ at $q$.
\end{theorem}

\begin{proof}
We make use again of Lemma~\ref{lem:mingeo} to conclude that $\dga(0)\in\norh(E,q)$. Lemma~\ref{lem:cones}(ii) implies that $\norh(E,q)=\{\rho \nu_q:\rho\ge 0\}$. Since both $\dga(0)$ and $\nu_q$ are unit vectors we obtain $\dga(0)=\nu_q$.
\end{proof}

Theorem~\ref{thm:exp} allows to describe geometrically  the metric projection and the distance function to some simple closed sets in $\hh^n$.

\begin{example}
\label{ex:vertplane}
Let $E:=\{x_1\le 0\}\subset\hh^n$. The tangent space at every point of every point of $S=\ptl E$ is generated by $\{Y_1,X_2,Y_2,\ldots,X_n,Y_n,T\}$. This implies that the outer unit normal $N$ is given by $X_1$, which coincides with $\nu_S$. Since $\escpr{N,T}\equiv 0$, Theorem~\ref{thm:exp} yields that the curvature of any minimizing geodesic leaving from $S$ is $0$. So minimizing geodesics are straight lines tangent to $X_1$. Given $p\not\in E$ with coordinates $(x_1,y_1,\ldots,x_n,y_n,t)$, a simple computation implies that the only point in $S$ at minimum distance from $p$ is $(0,y_1,\ldots,x_n,y_n,t-x_1y_1)$. Since
\[
(0,y_1,\ldots,x_n,y_n,t-x_1y_1)+x_1\,(1,0,\ldots,0,y_1)=(x_1,y_1,\ldots,x_n,y_n,t),
\]
we conclude $\delta_E(p)=x_1$. In this case the reach of $E$ is $+\infty$. Moreover, the distance function $\delta_E$ is $C^\infty$ (analytic) out of the set $E$.
\end{example}

\begin{example}[Behaviour near an isolated singular point]
\label{ex:isolated}
We denote by $(x,y,t)$ the coordinates in $\hh^1$. Let $E:=\{t\le 0\}\subset\hh^1$. The tangent plane at a boundary point $(x,y,0)$ is spanned by $X-yT$ and $Y+xT$, and so it is never horizontal unless $x=y=0$. Hence $E_0=\{(0,0,0)\}$. The outer unit normal to $S=\ptl E$ is given by
\[
N=\frac{yX-xY+T}{\sqrt{1+x^2+y^2}},
\]
and the horizontal unit normal on $S\setminus S_0$ by
\[
\nu_S=\frac{yX-xY}{\sqrt{x^2+y^2}}.
\]
Theorem~\ref{thm:exp} implies that the curvature of a minimizing geodesic starting from a point $(x_0,y_0,0)$, with $r_0^2=x_0^2+y_0^2\neq 0$ is given by
\[
2\,\frac{\escpr{N,T}}{\mnh}=\frac{2}{r_0}.
\]
The projection of this geodesic to the $xy$-plane is, by the discussion in \S~\ref{sub:geodesics}, a circle of radius $r_0/2$. Its initial velocity is given by the projection of $\nu_S$ to the $xy$ plane, and it is given by $y_0\tfrac{\ptl}{\ptl x}-x_0\tfrac{\ptl}{\ptl y}$. Such circles always contain the origin.

The horizontal liftings of these circles are horizontal geodesics containing points in the positive part of the $t$-axis. Any geodesic starting from a point $(x_0,y_0,0)$ in the circle $x_0^2+y_0^2=r_0^2$ with initial velocity $y_0\tfrac{\ptl}{\ptl x}-x_0\tfrac{\ptl}{\ptl y}$ reaches the $t$-axis after a distance $\pi r_0/2$ at the point $(0,0,\pi r_0^2/4)$ (twice half of the area of the disk of radius $r/2$). Observe that the point $(0,0,\pi r_0^2/4)$ is only reached by the geodesics described above, and so they are geodesics realizing the distance. Since an infinite number of geodesics reach this point, they are not minimizing in larger intervals. On the other hand, smaller segments minimize the distance to $E$. In particular, setting $t=\pi r_0^2/4$, we have $\pi r_0/2=(4t/\pi)^{1/2}(\pi/2)=\pi^{1/2}t^{1/2}$ and so
\[
\delta_E((0,0,t))=\pi^{1/2}t^{1/2}.
\]

To explicitly calculate the distance to the set $E$, we compute from \eqref{eq:geodesics} the geodesics with initial conditions $(x_0,y_0,0)$, $(A,B)=(y_0,-x_0)$ and curvature $2r_0^{-1}$ to obtain
\begin{align*}
x(s)&=\frac{1}{2}\big(x_0+y_0\sin(2r_0^{-1}s)+x_0\cos(2r_0^{-1}s)\big),
\\
y(s)&=\frac{1}{2}\big(y_0+y_0\cos(2r_0^{-1}s)-x_0\sin(2r_0^{-1}s)\big),
\\
t(s)&=\frac{r_0}{2}\big(s+\frac{r_0}{2}\sin(2r_0^{-1}s)\big).
\end{align*}
We know that $\delta_E((x(s),y(s),t(s)))=s$. The distance should only depend on $x^2+y^2$ and $t$. Indeed
\begin{align*}
x(s)^2+y(s)^2&=\frac{1}{2}\big(r_0^2+r_0^2\cos(2r_0^{-1}s)\big),
\\
t(s)&=\frac{r_0}{2}\big(s+\frac{r_0}{2}\sin(2r_0^{-1}s)\big).
\end{align*}
\end{example}

\begin{example}[Behaviour near a singular curve]
\label{ex:curve}
Let $E:=\{t\le xy\}\subset\hh^1$ be the subgraph of the function $u(x,y)=xy$, and $S=\ptl E$ be the graph of $u$. The tangent plane at every boundary point $(x,y,xy)$ is generated by $X$ and $Y+2xT$. Hence the outer unit normal and the corresponding horizontal unit normal are given by
\[
N=\frac{-2xY+T}{\sqrt{1+4x^2}},\qquad \nu_S=-\frac{x}{|x|}\,Y.
\]
The singular set is $S_0=\{(0,y,0):y\in\rr\}$. Its projection to the $xy$-plane is $x=0$. Given a point $p_0\in S\setminus S_0$ with coordinates $(x_0,y_0,x_0y_0)$, $x_0\neq 0$, Theorem~\ref{thm:exp} implies that the curvature of a minimizing geodesic leaving from $p$ is given by
\[
2\,\frac{\escpr{N,T}}{\mnh}(p_0)=\frac{1}{|x_0|}.
\]
So the projection of this geodesic to the $xy$-plane is a circle of radius $|x|$ with initial velocity $-\text{sgn}(x)\,\tfrac{\ptl}{\ptl y}$. We can compute from \eqref{eq:geodesics} the geodesics with initial conditions $(x_0,y_0,x_0y_0)$, $(A,B)=(0,-\text{sgn}(x_0))$ and curvature $|x_0|^{-1}$ to obtain
\begin{align*}
x(s)&=x_0\cos(|x_0|^{-1}s),
\\
y(s)&=y_0-x_0\sin(|x_0|^{-1}s),
\\
t(s)&=x_0y_0\cos(|x_0|^{-1}s)+|x_0|^{-1}x_0^2\,s.
\end{align*}
Such geodesics reach the plane with equation $x=0$ at time $s=\tfrac{\pi}{2}|x_0|$. For such a value, $x(s)=0$, $y(s)=y_0-x_0$ and $t(s)=\tfrac{\pi}{2}x_0^2$. Hence the point $(0,y_0-x_0,\tfrac{\pi}{2}x_0^2)$ is only reached by two geodesics of the same length: the one with initial condition $(x_0,y_0,x_0y_0)$, $B=(0,-\text{sgn}(x_0))$ and curvature $|x_0|^{-1}$ and the one with initial condition $(-x_0,y_0-2x_0,2x_0^2-x_0y_0)$, $(A,B)=(0,-\text{sgn}(-x_0))$ and curvature $|x_0|^{-1}$. This implies that both geodesics are minimizing.

Observe that $\delta_E((0,0,t))=(2\pi^{-1})^{1/2}t^{1/2}$.
\end{example}

\section{Regularity properties of the distance function to a submanifold}
\label{sec:distsm}

In this section we shall consider the regularity properties of the distance function to an $m$-dimensional submanifold $S\subset\hh^n$ of class $C^k$, with $k\ge 2$. Similar results in Euclidean spaces were obtained by Gilbarg and Trudinger \cite[\S~14.6]{MR1814364} and Hörmander \cite[p.~50]{MR0203075}. The reader is referred to Krantz and Parks monograph \cite[\S~4.4]{MR1894435} for historical background and references.

Given a point $q$ in a submanifold $S$ of class $C^1$, the tangent space $\ttan(S,q)$, as defined in the previous section, coincides with the classical tangent space $T_qS$ to the submanifold $S$. Hence a point $q$ is singular if and only if $T_qS\subset\hhh_q$. As usual, we denote the set of singular points in $S$ by $S_0$. The set $S_0$ is a closed subset of $S$. If $S\subset\hh^n$ is a hypersurface of class $C^1$, then the set $S_0$ coincides with the set of points where $T_qS=\hhh_q$.

\begin{lemma}
Let $S\subset \hh^n$ be an $m$-dimensional submanifold of class $C^1$. Let $p\not\in S$ and assume that $q\in S$ satisfies $\delta_S(p)=d(p,q)$. 

Then there exists a length-minimizing geodesic $\ga:[0,\delta(p)]\to\hh^n$ of curvature $\la$, parameterized by arc-length, joining $q$ and $p$ such that
\begin{enumerate}
\item[(i)] $\dga(0)\perp T_q S\cap\hhh_q$.
\item[(ii)] If $q\in S\setminus S_0$, the curvature $\la$ of $\ga$ is given by
\[
\la=\frac{2\escpr{E_q,\dga(0)}}{\escpr{E_q,T_q}},
\]
where $E_q\in T_qS$ is a unit vector orthogonal to $T_qS\cap\hhh_q$.
\end{enumerate}
\end{lemma}

\begin{proof}

Since $\ga$ is length-minimizing, equation \eqref{eq:1stvarlen} implies
\[
0=\escpr{u,\dga(0)+\frac{\la}{2}\,T_q}
\]
for any tangent vector $u\in T_qS$. This immediately implies (i). If $q\in S\setminus S_0$, then $\dga(0)$ is orthogonal to $T_qS\cap\hhh_q$. If $q$ is regular, there exists a unit vector $E_q\in T_qS$ orthogonal to $T_qS\cap\hhh_q$, unique up to sign. The above equation implies
\[
\la=-2\frac{\escpr{E_q,\dga(0)}}{\escpr{E_q,T_q}},
\]
which proves (ii).
\end{proof}

For a regular point $q\in S\setminus S_0$ and $v\in\hhh_q$ we define
\begin{equation}
\label{eq:deflambda}
\la(q,v):=-2\frac{\escpr{E_q,v}}{\escpr{E_q,T_q}}.
\end{equation}

For $q$ regular and $v$ horizontal and orthogonal to $T_qS\cap\hhh_q$, we define the map
\begin{equation}
\label{eq:exps}
\exp_S(q,v):=\ga_{q,v}^{\la(q,v)}(1),
\end{equation}
where $\ga_{q,v}^{\la(q,v)}=\ga$ is the sub-Riemannian geodesic of curvature $\la(q,v)$ and initial conditions $\ga(0)=q$, $\ga'(0)=v$.
When $v\neq 0$ is horizontal and orthogonal to $T_qS\cap\hhh_q$, we have
\[
\exp_S(q,v)=\ga_{p,v/|v|}^{\la(q,v/|v|)}(|v|),
\]
that coincides with the smooth geodesic starting from $q$ with initial speed $v/|v|$, parameterized by arc-length and (possibly) minimizing the distance to $S$, evaluated at $|v|$.

Our main result in this section is the following

\begin{theorem}
\label{thm:tube}
Let $S\subset\hh^n$ be a closed $m$-dimensional submanifold of class $C^k$, where $k\ge 2$ and $1\le m\le 2n$, and let $K\subset S\setminus S_0$ be a compact set. Then $\reach(S,K)>0$. Moreover, for $0<r<\reach(S,K)$, the distance function $\delta_S$ is of class $C^k$ in $(\xi_S^{-1}(K)\setminus K)\cap S_r$, where $S_r=\{p\in\hh^n:\delta_S(p)<r\}$.
\end{theorem}

For the proof of this result we shall need the following

\begin{lemma}[Tubular Neighborhood Lemma]
\label{lem:tube}
Let $S\subset\hh^n$ be a closed $m$-dimensio\-nal submanifold of class $C^k$, where $k\ge 2$ and $1\le m\le 2n$, and let $K\subset S\setminus S_0$ be a compact set. Let $Q=2n+1$. Then, for each point $q\in K$, there exists a neighborhood $U$ of $q$ in $S$ and an orthonormal family of $(Q-m)$ horizontal $C^{k-1}$ vector fields $X_i:U\to \hh^n$ such that the map $\Phi:U\times\rr^{Q-m}\to\hh^n$ defined by 
\begin{equation}
\label{eq:Phi}
\Phi(x,y)=\exp_S(x,\sum_{i=1}^{Q-m} y_iX_i)
\end{equation}
is a $C^{k-1}$ diffeomorphism in a neighborhood of $U\times\{0\}$.
\end{lemma}

\begin{proof}
We take a coordinate neighborhood $U'$ of $q$ in $S$ contained in the regular set and a family of $m$ vector fields on $U'$ of class $C^{k-1}$ that span the tangent space $T_xS$ for every $x\in U'$. This can be done using the Jacobian matrix of the immersion $U'\to\hh^n$. Projecting to the horizontal distribution, and using the Gram-Schmidt procedure we can find an orthonormal family
\[
Z_1,\ldots,Z_{m-1},N,X_1,\ldots,X_{Q-m},
\]
of vector fields of class $C^{k-1}$ so that $Z_1,\ldots,Z_{m-1},N$ span the tangent space to $S$ and $Z_1,\ldots,Z_{m-1}$ are horizontal.

On $U'\times\rr^{Q-m}$ we use formula \eqref{eq:Phi} to define the map $\Phi$, which is of class $C^{k-1}$ since the vector fields $X_i$ and the function $\la$ are of class $C^{k-1}$.

To apply the Implicit Function Theorem we compute the differential $d\Phi_{(q,0)}$ of the map $\Phi$ at $(q,0)$. For any vector $u\in T_qS$ we take a curve $\alpha$ defined in an open interval containing $0$ such that $\alpha(0)=q$, $\dot{\alpha}(0)=u$. Then
\[
d\Phi_{(q,0)}(u,0)=\frac{d}{ds}\bigg|_{s=0}\Phi(\alpha(s),0)=\frac{d}{ds}\bigg|_{s=0}\ga_{\alpha(s),0}^{\la(\alpha(s),0)}(1)=\dot{\alpha}(0)=u.
\]
On the other hand, if we take the coordinate vector $e_i$ in $\rr^{Q-m}$, for $i=1,\ldots,Q-m$, we have
\begin{align*}
d\Phi_{(q,0)}(0,e_i)&=\frac{d}{ds}\bigg|_{s=0}\Phi(q,(0,\ldots,\stackrel{(i)}s,\ldots,0))
\\
&=\frac{d}{ds}\bigg|_{s=0}\ga_{q,sX_i}^{\la(q,sX_i)}(1)
\\
&=
\frac{d}{ds}\bigg|_{s=0}\ga_{q,X_i}^{\la(q,X_i)}(s)=(X_i)_q.
\end{align*}
Hence $d\Phi_{(q,0)}$ is a linear isomorphism. The inverse function theorem then implies that $\Phi$ is a $C^{k-1}$ diffeomorphism in a neighborhood of $U\times\{0\}$ for some open neighborhood $U\subset U'$ of $q$.
\end{proof}

\begin{proof}[Proof of Theorem~\ref{thm:tube}]
Assume that $\reach(S,K)=0$. Then we can find sequences of points $p_i\in\hh^n\setminus S$, $q_i,q_i'\in S$ such that $q_i\neq q_i'$, and $\delta(p_i)=d(p_i,q_i)=d(p_i,q_i')\to 0$. Since the three sequences are bounded we can extract subsequences, denoted in the same way as the original sequence, converging to the same point $q\in K$.

Using Lemma~\ref{lem:tube} we can find a neighborhood $U$ of $q$ in $S$ so that $\Phi$ is a $C^{k-1}$ diffeomorphism of a neighborhood of $(p,0)$ in $S\times\rr^{Q-m}$ onto a neighborhood of $q$ in $\hh^n$. Since $p_i,q_i,q_i'$ converge to $q$, this is a contradiction to the injectivity of $\Phi$ (since $\Phi(q_i,y_i)=\Phi(q_i',y_i')=p_i$ for large $i$) that proves that $\reach(S,K)>0$.

Observe that the inverse function of $\Phi$ associates with every $p$, in the image of $\Phi$, the point $\xi(p)\in S$ at minimum distance from $p$ and the vector $y\in\rr^{Q-m}$ such that the geodesic with initial speed $\sum_{i=1}^{Q-m}y_i(p)(X_i)_{\xi(p)}$ connects $\xi(p)$ to $p$. Hence the distance $\delta(p)$ of $p$ to $\xi(p)$ equals $d(p,\xi(p))$, that is equal to $(\sum_{i=1}^{Q-m}y_i(p)^2)^{1/2}$. Hence $\delta(p)$ is of class $C^{k-1}$ whenever $y\neq 0$, i.e when $p\not\in S$. We have also that the map $\xi$ is of class $C^{k-1}$.

Now, for $p$ in the image of $\Phi$, define the function:
\[
v(p):=\sum_{i=1}^{Q-m}y_i(p)(X_i)_{\xi(p)},
\]
that it is of class $C^{k-1}$. Equation~\eqref{eq:1stvarlen} implies that the gradient of $\delta$ at the point $p$ is given by
\[
(\nabla\delta)_p=\dot{\ga}_{\xi(p),v(p)}^{\la(\xi(p),v(p))}(\delta(p))+\frac{\la(\xi(p),v(p))}{2}\,T_p.
\]
The function $\lambda$ is easily seen to be $C^{k-1}$ when $S$ is $C^k$ from its definition in \eqref{eq:deflambda}. Since $\xi(p)$, $\delta(p)$ and $v(p)$ are of class $C^{k-1}$, it follows that $\nabla\delta$ is of class $C^{m-1}$. Hence $\delta$ is of class $C^k$ as claimed.
\end{proof}

\begin{remark}
One could replace the notion of $m$-dimensional manifold by a suitable one of intrinsic $m$-dimensional submanifold. Two non-equivalent definitions were given by Franchi, Serapioni and Serra-Cassano \cite{MR2313532}.
\end{remark}

The following result extends the one by Arcozzi and Ferrari in $\hh^1$ \cite{MR2299576} to higher dimensional Heisenberg groups. It allows to slightly decrease the $C^2$ regularity hypothesis to obtain that a $C^{1,1}$ hypersurface has positive reach far from the singular set.  It provides many examples of sets of positive reach in $\hh^n$.

\begin{theorem}
\label{thm:reachc11}
Let $S\subset\hh^n$ be a closed $C^{1,1}$ hypersurface. Then, for any compact set $K\subset S\setminus S_{0}$, $\reach(S,K)>0$.
\end{theorem}

\begin{proof}

Consider two points $p_{0}$, $q_{0}$ in $\hh^n\setminus S$ at the same distance from $S$.  Let $p$, $q\in S$ points satisfying $\delta(p_0)=d(p,p_0)$ and $\delta(q_0)=d(q,q_0)$. Let $w_{p}:=\nu_{S}(p)$, $w_{q}:=\nu_{S}(q)$.  Let $\la_{p}$ and $\la_{q}$ the curvatures of unit speed minimizing geodesics joining $p$, $p_0$, and $q$, $q_0$, respectively.  Let $v_{p}$, $v_{q}$ the vectors in
$\rr^{2n}$ obtained from the coordinates of $w_{p}$ and $w_{q}$ in the basis $\{X_{i}, Y_{i} : i=1,\ldots, n\}$.   Let
\[\ga_{p}:=\ga_{{p},w_{p}}^{\la_{p}}, \qquad
\ga_{q}:=\ga_{{q},w_{q}}^{\la_{q}},
\]
and, for $a=p$, $q$, let $\alpha_{a}:=\pi\circ\ga_{a}$, $t_{a}:=t\circ\ga_{a}$. Then from \eqref{eq:geodesic2} we get
\begin{align*}
\alpha_{a}(s)&=\pi(a)+s\,\big(F(\la_{a}s)\,v_{a}+G(\la_{a}s)\,J(v_{a})\big) 
\\
t_{a}(s)&=t(a)+s^2\,H(\la_{a}s)+s\,G(\la_{a}s)\,\escpr{\pi(a),v_{a}}+
s\,F(\la_{a}s)\,\escpr{\pi(a),J(v_{a})}.
\end{align*}
From these expressions we obtain
\begin{align*}
|\pi(p)-\pi(q)|\le
|\alpha_{p}(s)-\alpha_{q}(s)|+s\,|F(\la_{p}s)&\,v_{p}-F(\la_{q}s)\,w_{q}|
\\
&+s\,|G(\la_{p}s)\,J(v_{p})-G(\la_{q}s)\,J(w_{q})|,
\end{align*}
and
\begin{align*}
|t(p)-t(q)|\le |t_{p}(s)-t_{q}(s)|&+s^2\,|H(\la_{p}s)-H(\la_{q}s)|
\\
&+s\,|G(\la_{p}s)\,\escpr{\pi(p),v_{p}}-G(\la_{q}s)\,\escpr{\pi(q),v_{q}}|
\\
&+s\,|F(\la_{p}s)\,\escpr{\pi(p),J(v_{p})}-F(\la_{q}s)\,\escpr{\pi(q),J(v_{q})}|
\end{align*}
Since we are considering minimizing geodesics we know that $|\la_{p} s|, |\la_{q}s|\le 2\pi$. As $F$, $G$, $H$ are Lipschitz in the interval $[-2\pi,2\pi]$, and $\la$, $\nu_{S}$ and $\pi$ are Euclidean Lipschitz, there exist positive constants $C_{i}$, $C_{i}'$ such that
\begin{align*}
|\pi(p)-\pi(q)|&\le
|\alpha_{p}(s)-\alpha_{q}(s)|+(C_{1}s+C_{2}s^2)\,|p-q|,
\\
|t(p)-t(q)|&\le
|t_{p}(s)-t_{q}(s)|+(C_{1}'s+C_{2}'s^2+C_{3}'s^3)\,|p-q|.
\end{align*}
Hence we get, for $s<1$,
\begin{equation*}
|\ga_{p}(s)-\ga_{q}(s)|\ge \frac{1}{\sqrt{2}}\,|p-q|\,(1-Cs),
\end{equation*}
for some constant $C>0$. This inequality implies that $S$ is locally of positive reach, since two minimizing geodesics cannot reach the same point for $s$ small enough.
\end{proof}

\begin{proposition}
\label{prop:parallelc11}
Let $E\subset\hh^n$ be a closed set. Let $K\subset\ptl E$ a compact set contained in the set of regular points of $E$. Assume that $\reach(E,K)>r_{0}>0$, and let $0<r<r_{0}$. Then $\ptl E_{r}\cap\xi^{-1}(K)$ is contained in a Euclidean $C^{1,1}$ hypersurface.
\end{proposition}

\begin{proof}
Let $p\in\ptl E_{r}\cap\xi^{-1}(K)$. Then $p\in\ptl B(\xi(p),r)$. Since $p$, $\xi(p)$ do not lie in the same vertical line, $p$ lies in the regular part of the boundary of $B(\xi(p),r)$.

Let $\ga:[0,r]\to\hh^n$ be the unique minimizing geodesic connecting $\xi(p)$ and $p$. We claim that $\ga$ is also minimizing in the larger interval $[0,r_{0}]$. To prove this, let $q\in \ptl E_{r_{0}}$ be the point in $\ptl E_{r_{0}}$ at minimum distance from $p$. Then
\[
d(q,\xi(p))\le d(q,p)+d(p,\xi(p))\le (r_{0}-r)+r=r_{0}.
\]
So $\xi(q)=\xi(p)$ and the only minimizing geodesic connecting $q$ and $\xi(p)$ is $\ga:[0,r_{0}]\to\hh^n$. Then $p\in\ptl B(q,r_{0}-r)$ and $p$ lies in the regular part of the boundary of $B(q,r_{0}-r)$.

Assume that $q\in\delta^{-1}(0,r_{0})\cap\xi^{-1}(K)$.  Then $\xi(q)$, $\la(q)$, $v(q)$ are continuous functions of $q$.  Hence the principal curvatures of $\ptl B(\xi(q),r)$ at $\ga_{\xi(q),v(q)}^{\la(q)}(r)$ depend on the second derivatives of $d_{\xi(q)}$ at the point $\ga_{\xi(q),v(q)}^{\la(q)}(r)$ and so they are continuous functions of $q$.  The same argument can be applied to the balls $B(\ga_{\xi(q),v(q)}^{\la(q)}(r_{0}),r_{0}-r)$. By Blaschke's Rolling Theorem, for every $q\in\ptl E_{r}\cap\xi^{-1}(K)$, there are two Euclidean balls of radius $R>0$ which are tangent at $q$ and leave
$\ptl E_{r}\cap\xi^{-1}(K)$ outside of the union of the balls. By Whitney's extension Theorem in Euclidean space, $\ptl E_{r}\cap\xi^{-1}(K)$ is contained in a Euclidean $C^1$ hypersurface. The condition on the tangent balls of uniform radius imply that $\ptl E_{r}\cap\xi^{-1}(K)$ is of Euclidean positive reach on both sides and so it is a Euclidean $C^{1,1}$ hypersurface.
\end{proof}

\begin{remark}
An interesting open question is the regularity of the distance function to a closed set $E\subset\hh^n$ when the boundary $\ptl E$ is of class $C_H^2$. This means that $\ptl E$ is locally the level set of a continuous function possessing horizontal derivatives of order two, see \cite{MR1871966}.
\end{remark}

\section{Steiner's formula for hypersurfaces
}
\label{sec:steiner}

Let $E\subset\hh^n$ be a closed set with $C^2$ boundary $S$. Consider a relatively compact open set $U$ in $S$ such that $\overline{U}\subset S\setminus S_0$. We know from Theorem~\ref{thm:tube} that $\reach(S,\overline{U})>0$ and that the distance function $\delta_E$ is of class $C^2$ in a neighborhood of $\overline{U}$ intersected with $\hh^n\setminus E$. For $r>0$ small enough, we want to find a formula expressing the volume of the set
\[
U_r:=\{p\in\hh^n\setminus E: \xi(p)\in U, \delta(p)\le r\}
\]
in terms of $r>0$ and the geometry of $S$.

The next Lemma, a version of the coarea formula, shows that it is enough to consider the sub-Riemanian area of the equidistant hypersurfaces. More general coarea formulas have been proven in the Heisenberg group and in more general spaces, e.g. Magnani \cite{MR1874099} and Karmanova \cite{MR2498573}.

\begin{lemma}
Let $E\subset\hh^n$ be a closed set with $C^2$ boundary $S$, and $U$ a relatively compact open set in $S$ such that $\overline{U}\subset S\setminus S_0$. Then, for $r>0$ small enough, we have
\begin{equation}
\label{eq:volumeur}
|U_r|=\int_0^r A(S_t\cap\xi^{-1}(U))\,dt,
\end{equation}
where $A$ is the sub-Riemannian area and $S_t$ is the hypersurface $S_t:=\{p\in\hh^n\setminus E: \delta(p)=t\}$.
\end{lemma}

\begin{proof}
In a neighborhood of $\overline{U}$ intersected with $\hh^n\setminus E$ the distance function $\delta$ is of class $C^2$ and has non-vanishing gradient by Theorem~\ref{thm:tube}. Hence, for $t>0$ small enough, the surface $S_t\cap\xi^{-1}(U)$ is a $C^2$ level set of $\delta$. By the Riemannian coarea formula
\[
|U_r|=\int_0^r\bigg\{\int_{S_t\cap\xi^{-1}(U)}\frac{1}{|\nabla\delta|}\,dS_t\bigg\}\,dt,
\]
where $dS_t$ is the Riemannian hypersurface area element on $S_t$. If $p\in S_t\cap\xi^{-1}(U)$, observe that we have
\[
(\nabla\delta)_p=\dot{\ga}_{\xi(p),\nuh(p)}^{\la(p)}(\delta(p))+\frac{\la(p)}{2}\,T_p.
\]
Since $\la(p)/2=(\escpr{N,T}/|N_h|)(\xi(p))$, where $N$ is the outer unit normal to $S$, we get
\[
|(\nabla\delta)_p|^2=\frac{1}{|(N_h)_{\xi(p)}|^2}.
\]
Then the formula \eqref{eq:volumeur} follows if we show
\begin{equation}
\label{eq:nhconstant}
|(N_h)_{\xi(p)}|=|(N^t_h)_p|,
\end{equation}
where $N^t$ is the outer unit normal to $S_t$, for all $t\in (0,r)$. To prove \eqref{eq:nhconstant} consider the geodesic $\ga:[0,\delta(p)]\to\hh^n$ joining $q=\xi(p)$ and $p$ with initial conditions $q$, $(\nuh)_q$ and curvature $	\la=2\escpr{N_q,T_q}/|(N_h)_q|$. For $t\in (0,\delta(p))$, the geodesic $\ga$ restricted to $[t,\delta(p)]$ also minimizes the distance to $S_t$ (this is a standard metric argument using the triangle inequality). Hence $\ga:[t,\delta(p)]\to\hh^n$ realizes the distance from $p$ to $S_t$ and we conclude that its curvature $\la$ equals $(\escpr{N^t,T_p}/|N^t_h|)(\ga(t))$. So we have 
\[
\frac{\escpr{N,T}}{|N_h|}(\xi(p))=\frac{\escpr{N^t,T}}{|N^t_h|}(\ga(t)),\quad\text{for all }t\in (0,\delta(p)).
\]
By continuity this formula also holds for $\ga(\delta(p))=p$, that implies \eqref{eq:nhconstant}. Hence we obtain
\[
\int_{S_t\cap\xi^{-1}(U)}\frac{1}{|\nabla\delta|}dS_t=\int_{S_t\cap\xi^{-1}(U)}|N^t_h|dS_t=A(S_t\cap\xi^{-1}(U)),
\]
that implies \eqref{eq:volumeur}.
\end{proof}

To obtain an expression for $|U_r|$ is then enough to compute the sub-Riemannian area of the parallel hypersurface $S_r\cap\xi^{-1}(U)$. We shall do it using the area formula. So let us consider the exponential map $\exp_S^t(q):=\ga_{q,\nuh(q)}^{\la(q)}(t)$ restricted to $U$, $\exp^t_S:U\to S_t\cap\xi^{-1}(U)$, and take into account that
\[
A(S_t\cap\xi^{-1}(U))=\int_{S_t\cap\xi^{-1}(U)}|N^t_h|dS_t
=\int_U|N_h|\,\text{Jac}(\exp_S^t)dS,
\]
where $dS$ and $dS_t$ are the Riemannian volume elements in $S$ and $S_t$, respectively, and $\text{Jac}(\exp_S^t)$ is the Jacobian of the map $\exp_S^t$.

We compute the Jacobian of $\exp_S^t$ the following way: we fix $q\in S$ and an orthonormal basis $e_1,\ldots,e_{2n}$ of $T_qS$. Let $\ga$ be the geodesic with initial conditions $q$, $(\nuh)_q$, and curvature $2\escpr{N_q,T_q}|(N_h)_q|$. The vector fields $(d\exp_S^t)_q(e_i)$, $i=1,\ldots,2n$, along $\ga$ are the Jacobi fields $E_i(t)$ along the geodesic $\ga(t)$ with initial conditions $E_i(0)=e_i$, $\dot{E}_i(0)=\nabla_{e_i}\nuh+2\escpr{J(\dga(0)),e_i}T_q$, and derivative of curvature given by $2e_i(\escpr{N,T}/|N_h|)$. Choosing, for any $i$, a curve $\alpha_i:(-\eps,\eps)\to S$ satisfying $\alpha_i(0)=q$ and $\dot{\alpha}_i(0)=e_i$, it is easy to check that $E_i(t)$ is the vector field $\ptl/\ptl s|_{s=0} \exp^t_S(\alpha(s))$ associated to the variation $(s,t)\mapsto \exp_S^t(\alpha(s))$ by arc-length parameterized geodesics. Observe that any Jacobi field $U$ along $\ga$ corresponding to a variation by arc-length parameterized curves and such that $U(0)\in T_qS$ satisfies
\begin{equation}
\label{eq:dga+t}
\escpr{U,\dga+\tfrac{\la}{2}T}=0
\end{equation}
along $\ga$, where $\la=2\escpr{N_q,T_q}/|(N_h)_q|$. To prove \eqref{eq:dga+t} we recall that the function $\rho(t):=\escpr{U(t),\dga(t)+\tfrac{\la}{2}T_{\ga(t)}}$ is constant by Remark~\ref{rem:unitspeedjacobi}. Since $\dga(0)+\tfrac{\la}{2}T_{\ga(0)}=N_q$, where $N$ is a Riemannian unit normal to $S$, we get $\rho(0)=\escpr{U(0),N_q}=0$.

The Jacobian is then given by
\[
\text{Jac}(\exp_S^t)(q)=|E_1(t)\wedge\ldots\wedge E_{2n}(t)|.
\]

If $n\ge 2$, we take an orthonormal family $X_2,Y_2,\ldots,X_n,Y_n$ of horizontal left-invariant vector fields, orthogonal to $\dga$ and $J(\dga)$, and such that $Y_i=J(X_i)$. We consider the orthonormal basis along $\ga$ given by the vectors $\dga,J(\dga),T,X_2,Y_2,\ldots,X_n,Y_n$. Because of equality \eqref{eq:dga+t}, we can express $E_1\wedge\ldots\wedge E_{2n}$ along $\ga$ as a linear combination of  $J(\dga)\wedge T\wedge X_2\wedge\ldots\wedge Y_n$ and $\dga\wedge J(\dga)\wedge X_2\wedge\ldots\wedge Y_n$. Again by equality \eqref{eq:dga+t}, we get that $\text{Jac}(\exp_S^t)(q)$ is equal to $(1+(\tfrac{\la(q)}{2})^2)^{1/2}=1/|N_h|(q)$ times the modulus of the determinant of the matrix 
\begin{equation}
\label{eq:matrix}
B(t)={\setlength{\arraycolsep}{0pt}\begin{pmatrix}
\escpr{E_1,J(\dga)} & \escpr{E_1,T} & \escpr{E_1,X_2}  & \escpr{E_1,Y_2} &\ldots  & \escpr{E_1,X_n} & \escpr{E_1,Y_n}
\\
\vdots & \vdots & \vdots& \vdots & \ddots & \vdots & \vdots &
\\
\escpr{E_i,J(\dga)} & \escpr{E_i,T} & \escpr{E_i,X_2}  & \escpr{E_i,Y_2} & \ldots  & \escpr{E_i,X_n} & \escpr{E_i,Y_n}
\\
\vdots & \vdots &  \vdots & \vdots & \ddots & \vdots 
\\
\escpr{E_{2n},J(\dga)} & \escpr{E_{2n},T} & \escpr{E_{2n},X_2}  & \escpr{E_{2n},Y_2} & \ldots  & \escpr{E_{2n},X_{2n}} & \escpr{E_{2n},Y_{2n}}
\end{pmatrix}}
\end{equation}
evaluated at the point $\ga(t)$. Hence we get
\begin{equation}
\label{eq:areadetb}
A(S_t\cap\xi^{-1}(U))=\int_U |\det(B(t))|\,dS.
\end{equation}

If $n=1$, reasoning as in the previous paragraph we obtain that $\dga\wedge J(\dga)$ is a linear combination of $J(\dga)\wedge T$ and $\dga\wedge J(\dga)$, and that the Jacobian is given by $|N_h|^{-1} |\det(B(t))|$, where $B(t)$ is now the matrix
\begin{equation}
\label{eq:matrix1}
\begin{pmatrix}
\escpr{E_1,J(\dga)} & \escpr{E_1,T}
\\
\escpr{E_2,J(\dga)} & \escpr{E_2,T}
\end{pmatrix}
\end{equation}
evaluated at the point $\ga(t)$. Again the expression for $A(S_t\cap\xi^{-1}(U))$ is given by formula \eqref{eq:areadetb}.

\subsection{The case of $\hh^1$}

For the first Heisenberg group we have the following result

\begin{theorem}
\label{thm:steiner1}
Let $S\subset\hh^1$ be a hypersurface of class $C^k$, $k\ge 2$, bounding a closed region $E$, and let $U\subset S$ be an open subset such that $\overline{U}\subset S\setminus S_0$. Then the volume of the one-side tubular neighborhood $U_r=\{p\in\hh^1:\xi_E(p)\in U, \de(p)<r\}$ is given by
\begin{equation}
\label{eq:tub1}
|U_r|= \sum_{i=0}^4\int_U\bigg\{\int_0^r a_i f_i(\la, s)\,ds\bigg\}\,dS,
\end{equation}
where $\la$ is the function $2\escpr{N,T}/|N_h|$, defined on $S\setminus S_0$, the functions $f_i$ have been defined in \eqref{eq:fi}, and the  coefficients $a_i$ are given by the expressions
\begin{equation}
\label{eq:ai}
\begin{split}
a_{0}&=|N_{h}|,
\\
a_{1}&=|N_h| H,
\\
a_{2}&=-4|N_h|e_1\bigg(\ntnh\bigg),
\\
a_{3}&=-4e_2\bigg(\ntnh\bigg),
\\
a_{4}&=-4He_2\bigg(\ntnh\bigg)-4|N_h|\bigg(e_1\bigg(\ntnh\bigg)\bigg)^2.
\end{split}
\end{equation}

In the formulas \eqref{eq:ai}, $H$ is the mean curvature of $S\setminus S_0$, $e_1=J(\nuh)$ and $e_2=\escpr{N,T}\nuh-\mnh T$.
\end{theorem}

The rest of this subsection is devoted to the proof of Theorem~\ref{thm:steiner1}. We shall use formula \eqref{eq:areadetb} and the expression of $B(t)$ given in \eqref{eq:matrix1}. For the first part of these computation we shall consider an arbitrary orthonormal basis $\{e_1,e_2\}$ of the tangent plane $T_qS$. Consider the functions $c_i=\escpr{E_i,T}$, $i=1,2$, defined along $\ga$. By \eqref{eq:jacobiequations} we have $2\escpr{E_i,J(\dga)}=\dot{c}_i$,$i=1,2$, so that
\[
A(S_t\cap\xi^{-1}(U))=\frac{1}{2}\int_U |\dot{c}_1c_2-c_1\dot{c}_2|\, dS.
\]
To compute $c_{i}(s)$, $i=1,2$, we shall use that $c_{i}$ satisfies the third order equation $\dddot{c}_{i}+\la^2\dot{c}_{i}+2\la_{i}'=0$, as shown in Corollary~\ref{cor:jacobiequations-1}, where $\la$ is the function $2\escpr{N,T}/|N_{h}|$, and $\la_{i}'$ is equal to $e_{i}(\la)$ for $i=1,2$. By \eqref{eq:expc}, the solutions of this equation are given by
\[
c_i(s)=c_i(0)+\dot{c}_i(0)f_1(\la, s)+\ddot{c}_i(0)f_2(\la, s)-2\la_i'k(\la, s),
\]
where the functions $f_1, f_2, k$ were defined in \eqref{eq:fi}. 

So we get
\begin{align*}
\dot{c}_i(s)&=\dot{c}_i(0)f_0(\la s)+\ddot{c}_i(0)f_1(\la, s)-2\la_i'f_2(\la, s).
\end{align*}
Hence
\begin{align*}
(c_{1}\dot{c}_{2}-c_{2}\dot{c}_{1})(t)=&\,(c_{1}\dot{c}_{2}-c_{2}\dot{c}_{1})(0)\,f_{0}(\la, s)
\\
+&\,(c_{1}\ddot{c}_{2}-c_{2}\ddot{c}_{1})(0)\,f_{1}(\la, s)
\\
+&\,(\dot{c}_{1}\ddot{c}_{2}-\dot{c}_{2}\ddot{c}_{1})(0)\,(f_{1}^2-f_{0}f_{2})(\la, s)
-\,2\,(c_{1}(0)\la_2'-c_{2}(0)\la_1')\,f_{2}(\la, s)
\\
-&\,2\,(\dot{c}_{1}(0)\la_2'-\dot{c}_{2}(0)\la_1')\,(f_{1}f_{2}-f_{0}k)(\la, s)
\\
-&\,2\,(\ddot{c}_{1}(0)\la_2'-\ddot{c}_{2}(0)\la_1')\,(f_{2}^2-f_{1}k)(\la, s).
\end{align*}
A simple computation shows that
\begin{equation*}
(f_{1}^2-f_{0}f_{2})(\la, s)=f_{2}(\la,s),
\end{equation*}
and so we get
\begin{align*}
\tfrac{1}{2}(c_{1}\dot{c}_{2}-c_{2}\dot{c}_{1})(s)=&\tfrac{1}{2}(c_{1}\dot{c}_{2}-c_{2}\dot{c}_{1})(0)\,f_{0}(\la,s)
\\
+&\tfrac{1}{2}(c_{1}\ddot{c}_{2}-c_{2}\ddot{c}_{1})(0)\,f_{1}(\la, s)
\\
+&\tfrac{1}{2}\big((\dot{c}_{1}\ddot{c}_{2}-\dot{c}_{2}\ddot{c}_{1})(0)-2\,(c_{1}(0)\la_{2}'-c_{2}(0)\la_{1}')\big)\,f_{2}(\la,s)
\\
-&(\dot{c}_{1}(0)\la_{2}'-\dot{c}_{2}(0)\la_{1}')\,f_3(\la, s)
\\
-&(\ddot{c}_{1}(0)\la_{2}'-\ddot{c}_{2}(0)\la_{1}')\,f_4(\la, s).
\end{align*}

Let us remark that the functions $f_0,f_1,f_2,f_3,f_4$ have the order indicated by their subscripts. Let us call $a_{i}$ to the coefficients appearing in the above formula, that is
\begin{equation*}
\begin{split}
a_0&=\tfrac{1}{2}(c_{1}\dot{c}_{2}-c_{2}\dot{c}_{1})(0),
\\
a_1&=\tfrac{1}{2}(c_{1}\ddot{c}_{2}-c_{2}\ddot{c}_{1})(0),
\\
a_2&=\tfrac{1}{2}\big((\dot{c}_{1}\ddot{c}_{2}-\dot{c}_{2}\ddot{c}_{1})(0)-2\,(c_{1}(0)\la_{2}'-c_{2}(0)\la_{1}')\big),
\\
a_3&=-(\dot{c}_{1}(0)\la_{2}'-\dot{c}_{2}(0)\la_{1}'),
\\
a_4&=-(\ddot{c}_{1}(0)\la_{2}'-\ddot{c}_{2}(0)\la_{1}').
\end{split}
\end{equation*}
Hence we have
\begin{equation*}
\tfrac{1}{2}(c_{1}\dot{c}_{2}-c_{2}\dot{c}_{1})(s)=a_{0}f_{0}(\la,s)+a_{1}f_{1}(\la,s)+a_{2}f_{2}(\la,s)+a_{3}f_3(\la,s)+a_{4}f_4(\la,s).
\end{equation*}

We now compute the coefficients $c_{i}(0),\dot{c}_{i}(0),\ddot{c}_{i}(0)$ and $\la_{i}'$, for $i=1,2$. At this point we choose as orthonormal basis on $T_qS$ the one determined by the vectors $e_1=J((\nuh)_q)$ and $e_2=(\escpr{N,T}\nuh-|N_h|T)_q$. We have
\begin{equation}
\label{eq:ci}
c_{i}(0)=\begin{cases}
0, & i=1, \\
-|(N_{h})_{q}|, & i=2.
\end{cases}
\end{equation}
For the first derivative take into account that, for a generic Jacobi field $U$ along the geodesic $\ga$ we get
\begin{equation}
\label{eq:ci'0}
\frac{d}{ds}\bigg|_{s=0}
\escpr{U(s),T_{\ga(s)}}=\escpr{\nabla_{\dga(0)}U,T_{q}}=2\tor(J(\dga(0)),U(0))=2\,\escpr{e_{1},U(0)},
\end{equation}
so that
\begin{equation}
\label{eq:ci'}
\dot{c}_{i}(0)=\begin{cases}
2, & i=1,
\\
0, & i=2.
\end{cases}
\end{equation}
For the second derivative we obtain from the Jacobi equation \eqref{eq:jacobi}
\begin{align*}
\frac{d^2}{ds^2}
\escpr{U,T}&=\dga\escpr{\nabla_{\dga} U,T}=\dga(\tor(\dga,U))
\\&=2\dga\escpr{J(\dga),U}
\\
&=2\,\big(\escpr{J(\nabla_{\dga}\dga),U}+\escpr{J(\dga),\nabla_U\dga}\big)
\\
&=2\,\big(\la\escpr{\dga,U}+\escpr{J(\dga),\nabla_U\dga}\big).
\end{align*}
Evaluating at $s=0$ we get
\begin{equation}
\label{eq:ci''0}
\frac{d^2}{ds^2}\bigg|_{s=0}\escpr{U(s),T_{\ga(s)}}=2\la\escpr{(\nuh)_q,U(0)}+2\escpr{e_1,\nabla_{U(0)}\nuh}.
\end{equation}
Hence we obtain
\begin{equation*}
\ddot{c}_{i}(0)=\begin{cases}
2\escpr{e_1,\nabla_{e_1}\nuh}, & i=1,
\\
2\la\escpr{N_q,T_q}+2\escpr{e_1,\nabla_{e_2}\nuh},
&i=2.
\end{cases}
\end{equation*}
When the surface is of class $C^2$, its sub-Riemannian mean curvature, see \cite{MR2435652} and \cite{MR2898770}, is defined by
\[
H(q)=\escpr{e_1,\nabla_{e_1}\nuh}.
\]
The covariant derivative of $\nuh$ in the direction of $e_2$ was computed in \eqref{eq:nablaEnuh}. Its product with $e_1$ is given by
\[
\escpr{\nabla_{e_2}\nuh,e_1}=-|(N_h)_q|\,\bigg(e_1\bigg(\ntnh\bigg)+2\bigg(\ntnh\bigg)^2(q)\bigg).
\]
Hence we get
\begin{equation}
\label{eq:ci''}
\ddot{c}_{i}(0)=\begin{cases}
2H(q), & i=1,
\\
-2|(N_h)_q|\,e_1\big(\ntnh\big),
&i=2.
\end{cases}
\end{equation}
Finally, we notice that
\begin{equation}
\label{eq:lai'}
\la_{i}'=e_i(\la)=2e_i\bigg(\frac{\escpr{N,T}}{|N_h|}\bigg), \quad i=1,2.
\end{equation}

From equations \eqref{eq:ci}, \eqref{eq:ci'}, \eqref{eq:ci''} and \eqref{eq:lai'} we finally obtain \eqref{eq:ai}.

\begin{remark}
The functions $f_i$ are analytic and can be written as power series. This way we can obtain an expression for $|U_r|$ of any order. To obtain an expansion of order three, using \eqref{eq:fi}, we compute
\begin{align*}
f_0(\la,s)&=\cos(\la s)=1-\frac{\la^2s^2}{2}+o(s^3),
\\
f_1(\la,s)&=s+o(s^3),
\\
f_2(\la,s)&=\frac{1}{2}s^2+o(s^3).
\end{align*}
So, using \eqref{eq:tub1} and \eqref{eq:ai}, we have
\begin{multline}
\label{eq:power3}
|U_r|=A(U)\,r+\frac{1}{2}\bigg(\int_UHdP\bigg)r^2
\\
-\frac{2}{3}\bigg(\int_U\bigg\{e_1\bigg(\ntnh\bigg)+\bigg(\ntnh\bigg)^2\bigg\}dP\bigg)r^3+o(r^4),
\end{multline}
where $A(U)$ is the sub-Riemannian area of $U$ and $dP$ is the sub-Riemannian perimeter measure on $S$, defined as $|N_h|dS$ ($dS$ is the Riemannian measure on $S$).
\end{remark}

\subsection{The case of $\hh^n$, $n\ge 2$}

For higher dimensional Heisenberg groups we have

\begin{theorem}
\label{thm:steiner2}
Let $S\subset\hh^n$, $n\ge 2$, be a hypersurface of class $C^k$, $k\ge 2$, and let $U\subset S$ be an open subset such that $\overline{U}\subset S\setminus S_0$. Then the volume of the tubular neighborhood $U_r=\{p\in\hh^1:\xi_S(p)\in U, d(p,S)<r\}$ is given by
\begin{equation}
|U_r|=\int_U\bigg\{ \int_0^r |\det(B(s))|\,ds\bigg\}\,dS,
\end{equation}
where $B(s)$ is the matrix in \eqref{eq:matrix}. The function $|\det(B(s))|$ is an analytic function of $\la$ and $s$ multiplied by coefficients involving $\escpr{N,T}/\mnh$, $\mnh$, the horizontal gradient in $S$ of the function $\escpr{N,T}/\mnh$ and the principal curvatures of the horizontal second fundamental form.
\end{theorem}

The proof of Theorem~\ref{thm:steiner2} was given at the beginning of Sestion~\ref{sec:steiner}. We now make a choice of the orthonormal basis $e_i$ of $T_qS$.

For fixed $q\in U$, we take an orthonormal basis $(e_1,\ldots,e_{2n})$  of $T_qS$ so that $e_1=J((\nuh)_q)$, $e_2=(\escpr{N,T}\nuh-|N_h|T)_q$, and the remaining vectors $e_j$, $j>2$, are chosen so that
\begin{align*}
(e_3,e_4,\ldots,e_{2n-1},e_{2n})&=(e_3,J(e_3),\ldots,e_{2n-1},J(e_{2n-1}))
\\
&=((X_2)_q,(Y_2)_q,\dots,(X_n)_q,(Y_n)_q).
\end{align*}
With this notation, we have $e_{2i-1}^\ell=X_i$ and $e_{2i}^\ell=Y_i$ for all $i\ge 2$.

We shall consider the Jacobi fields $E_i$, $i=1,\ldots,2n$ along the geodesic $r\mapsto \exp_S(q,r(\nuh)_q)$ satisfying $E_i(0)=e_i$, $\dot{E}_i(0)=\nabla_{e_i}\nuh+2\escpr{J(\dga(0)),e_i}\,T_q$ and $\la_i'=e_i(2\escpr{N,T}/\mnh)$. We let $c_i=\escpr{E_i,T}$.

Using the third equation in \eqref{eq:jacobiequations}, we rewrite the first two columns of the matrix $B$ so that

\begin{equation}
\label{eq:matrix-2}
B={\setlength{\arraycolsep}{3pt}\begin{pmatrix}
\tfrac{1}{2}\dot{c}_1 & c_1 & \escpr{E_1,X_2}  & \escpr{E_1,Y_2} &\ldots  & \escpr{E_1,X_n} & \escpr{E_1,Y_n}
\\
\tfrac{1}{2}\dot{c}_2 & c_2 & \escpr{E_2,X_2}  & \escpr{E_2,Y_2} &\ldots  & \escpr{E_2,X_n} & \escpr{E_2,Y_n}
\\
\vdots & \vdots & \vdots& \vdots & \ddots & \vdots & \vdots &
\\
\tfrac{1}{2}\dot{c}_i & c_i & \escpr{E_i,X_2}  & \escpr{E_i,Y_2} & \ldots  & \escpr{E_i,X_n} & \escpr{E_i,Y_n}
\\
\vdots & \vdots &  \vdots & \vdots & \ddots & \vdots 
\\
\tfrac{1}{2}\dot{c}_{2n} & c_{2n} & \escpr{E_{2n},X_2}  & \escpr{E_{2n},Y_2} & \ldots  & \escpr{E_{2n},X_{2n}} & \escpr{E_{2n},Y_{2n}}
\end{pmatrix}}.
\end{equation}

The first two columns of the matrix $B(s)$ can be compued from Corollary~\ref{cor:jacobiequations-1} since
\begin{align*}
c_i(s)&=c_i(0)+\dot{c}_i(0)f_1(\la, s)+\ddot{c}_i(0)f_2(\la, s)-2\la'_ik(\la, s),
\\
\dot{c}_i(s)&=\dot{c}_i(0)f_0(\la,s)+\ddot{c}_i(0)f_1(\la,s)-2\la'_i f_2(\la,s).
\end{align*}
The coefficients $c_i(0),\dot{c}_i(0),\ddot{c}_i(0)$ can be obtained as in the case of the first Heisenberg group, using formulas \eqref{eq:ci'0} and \eqref{eq:ci''0}, to get
\begin{align*}
c_i(0)&=0, i\neq 2, \qquad c_2(0)=-\mnh,
\\
\dot{c}_i(0)&=0, i\neq 1, \qquad \dot{c}_1(0)=2,
\end{align*}
and
\begin{align*}
\ddot{c}_i(0)&=2\escpr{e_1,\nabla_{e_i}\nuh},\quad i\neq 2,
\\
\ddot{c}_2(0)&=-2\mnh e_1\big(\tfrac{\escpr{N,T}}{\mnh}\big).
\end{align*}

The remaining columns can be calculated from the expression for the horizontal component of a Jacobi field given in Lemma~\ref{lem:jacobi} and the fact that $X_j$, $Y_j$, $j\ge 2$, are orthogonal to $\dga$ and $J(\dga)$ along $\ga$. From equation \eqref{eq:explicitUh} we obtain the equality
\begin{multline*}
(E_i)_h(s)=(e_i^\ell)_{\ga(s)}+f_1(\la,s)(\dot{e}_i)^\ell_{\ga(s)}-\la f_2(\la,s)J(\dot{e}_i)^\ell_{\ga(s)}
\\
+e_i(\la)\big[k(\la,s)\dga(s)+\la f_2(\la,s)J(\dga(s))\big], \quad i= 1,\ldots,2n,
\end{multline*}
where $\dot{e}_i=\nabla_{e_i}\nuh$, and so
\begin{equation}
\label{eq:Bj}
\begin{split}
\escpr{E_i,X_j}&=\escpr{e_i,e_{2j-1}}+f_1(\la,s)\escpr{\dot{e}_i,e_{2j-1}}-\la f_2(\la,s)\escpr{J(\dot{e}_i),e_{2j-1}},
\\
\escpr{E_i,Y_j}&=\escpr{e_i,e_{2j}}+f_1(\la,s)\escpr{\dot{e}_i,e_{2j}}+\la f_2(\la,s)\escpr{J(\dot{e}_i),e_{2j}},
\end{split}
\end{equation}
for all $i=1,\ldots,2n$ and $j>2$. The above formulas follow since $\escpr{(e_i^\ell)_p,(e_j^\ell)_p}=\escpr{e_i,e_j}$ for all $i,j$ and $p\in\hh^n$.

We now use these computations to calculate the series development of $|U_r|$ up to order three. First we notice that $|\det(B(s))|=-\det(B(s))$ for $s>0$ small enough. The derivatives of the function $s\mapsto\det(B(s))$ at $s=0$ will be obtained writing the matrix $B(s)$ as a function of their columns $B^1(s),\ldots,B^{2n}(s)$, and using the classical formulas
\begin{equation*}
\frac{d}{ds}\det(B(s))=\sum_{i=1}^{2n} \det(B^1,\ldots, \dot{B}^i,\ldots,B^{2n})(s),
\end{equation*}
and
\begin{multline*}
\frac{d^2}{ds^2}\det(B(s))=\sum_{i=1}^{2n}\det(B^1,\ldots,\ddot{B}^i,\ldots,B^{2n})(s)
\\
+\sum_{i\neq j} \det(B^1,\ldots,\dot{B}^i,\ldots,\dot{B}^j,\ldots,B^{2n})(s).
\end{multline*}

The first column of the matrix $B$ and their derivatives up to order two are computed from formulas \eqref{eq:ci'0}, \eqref{eq:ci''0} and equation \eqref{eq:odec}:
\[
B^1(0)=\begin{pmatrix}
1 \\ 0 \\ 0 \\ \vdots \\ 0
\end{pmatrix},\quad
\dot{B}^1(0)=\begin{pmatrix}
\escpr{e_1,\nabla_{e_1}\nuh} \\ -\mnh e_1\big(\tfrac{\escpr{N,T}}{\mnh}\big) 
\\ \escpr{e_1,\nabla_{e_3}\nuh} \\ \vdots \\ \escpr{e_1,\nabla_{e_{2n}}\nuh}
\end{pmatrix},\qquad
\ddot{B}^1(0)=\begin{pmatrix}
-\la^2-e_1(\la) \\ -e_2(\la) \\ -e_3(\la)\\ \vdots \\ -e_{2n}(\la)
\end{pmatrix}.
\]
The second column and their derivatives are computed from \eqref{eq:ci'0} and \eqref{eq:ci''0}:
\[
B^2(0)=\begin{pmatrix} 0 \\ -\mnh \\ 0 \\ \vdots \\ 0 \end{pmatrix}, \quad
\dot{B}^2(0)=\begin{pmatrix} 2 \\ 0 \\ 0 \\ \vdots \\ 0 \end{pmatrix}, \quad
\ddot{B}^2(0)=\begin{pmatrix} 2\escpr{e_1,\nabla_{e_1}\nuh} \\ -2\mnh e_1\big(\tfrac{\escpr{N,T}}{\mnh}\big) 
\\ 2\escpr{e_1,\nabla_{e_3}\nuh} \\ \vdots \\ 2\escpr{e_1,\nabla_{e_{2n}}\nuh}
\end{pmatrix}
\]
The remaining columns $B^i$, $i\ge 3$, and their derivatives at $s=0$ are computed from \eqref{eq:Bj}:
\[
B^i(0)=\begin{pmatrix} 0 \\ \vdots \\ \stackrel{(i)}{1} \\ \vdots\\ 0 \end{pmatrix}, \quad
\dot{B}^i(0)=\begin{pmatrix} \escpr{\nabla_{e_1}\nuh,e_i} \\ \vdots \\ \escpr{\nabla_{e_{2n}}\nuh,e_i}\end{pmatrix},\quad
\ddot{B}^i(0)=\begin{pmatrix} -\tfrac{\la}{2}\escpr{J(\nabla_{e_1}\nuh),e_i} \\ \vdots \\ -\tfrac{\la}{2}\escpr{J(\nabla_{e_{2n}}\nuh),e_i}\end{pmatrix}.
\]

Hence we get
\[
\frac{d}{ds}\bigg|_{s=0}\det(B(s))=-\mnh\sum_{i\neq 2}\escpr{\nabla_{e_i}\nuh,e_i}=-\mnh H,
\]
where $H$ is the mean curvature of $S$, and
\begin{align}
\frac{d^2}{ds^2}\bigg|_{s=0}\!\!\det(B(s))&=\mnh(\la^2+e_1(\la))-2\mnh e_1\bigg(\ntnh\bigg)\notag
\\
\label{eq:detb''}
&\qquad-\mnh\sum_{i=3}^{2n}\frac{\la}{2}\escpr{J(\nabla_{e_i}\nuh),e_i}
\\
&+4\mnh e_1\bigg(\ntnh\bigg)\notag
\\
&-\mnh\sum_{i,j\neq 2, i\neq j}\!\!\big(\escpr{\nabla_{e_i}\nuh,e_i}\escpr{\nabla_{e_j}\nuh,e_j}-\escpr{\nabla_{e_i}\nuh,e_j}\escpr{\nabla_{e_j}\nuh,e_i}\big).\notag
\end{align}
Adding the first and third lines in \eqref{eq:detb''} we obtain
\[
4\mnh\bigg(e_1\bigg(\ntnh\bigg)+\bigg(\ntnh)\bigg)^2\bigg).
\]
To treat the second line in \eqref{eq:detb''} we notice that the quantity $\sum_{i=3}^{2n}\escpr{J(\nabla_{e_i}\nuh,e_i}$ is the trace of the bilinear form $(v,w)\mapsto \escpr{J(\nabla_v\nuh),w}$ in the subspace $TS\cap\hhh$ (since $\escpr{J(\nabla_{e_1}\nuh,e_1}=0$). Hence it can be computed using any orthonormal basis in $TS\cap\hhh$. Taking one basis composed of principal directions $v_i, i=1,\ldots,(2n-1)$, for the horizontal second fundamental form we obtain
\[
\sum_{i=3}^{2n}\escpr{J(\nabla_{e_i}\nuh),e_i}=\sum_{i=1}^{2n-1} \ntnh\escpr{J(v_i)_{ht},J(v_i)}=(2n-2)\ntnh.
\]
To treat the last line in \eqref{eq:detb''} we first notice that the terms corresponding to $i=j$ can be added since they all vanish. This way we obtain
\[
\sum_{i,j\neq 2}\escpr{\nabla_{e_i}\nuh,e_i}\escpr{\nabla_{e_j}\nuh,e_j}-\sum_{i,j\neq 2}\escpr{\nabla_{e_i}\nuh,e_j}\escpr{\nabla_{e_j}\nuh,e_i}.
\]
The first sum is just the squared mean curvature $\big(\sum_{i\neq 2}\escpr{\nabla_{e_i}\nuh,e_i}\big)^2$. The second one can be expressed as
\[
\frac{1}{2}\sum_{i,j\neq 2}\big(\escpr{\nabla_{e_i}\nuh,e_j}+\escpr{\nabla_{e_j}\nuh,e_i}\big)^2-\sum_{i\neq 2}|\nabla_{e_i}\nuh|^2.
\]
Both quantities are independent of the orthonormal basis chosen. The first one since it is the squared norm of the symmetric bilinear form $(v,w)\mapsto \escpr{\nabla_v\nuh,w}+\escpr{\nabla_w\nuh,v}$. The second one is just the squared norm of the horizontal second fundamental form. If we choose an orthonormal basis of principal directions we get the value $|\sg|^2$. Hence the last line in \eqref{eq:detb''} is equal to
\[
\mnh\big(-H^2+|\sg|^2\big).
\]

In summary,
\[
\frac{d^2}{ds^2}\det(B(s))\bigg|_{s=0}\!\!\!\!=\mnh\bigg(4e_1\bigg(\ntnh\bigg)+(2n+2)\bigg(\ntnh\bigg)^2+|\sg|^2-H^2\bigg).
\]

We have thus proved the following result

\begin{theorem}
\label{thm:steiner23}
Let $S\subset\hh^n$, $n\ge 2$, be a hypersurface of class $C^k$, $k\ge 2$, and let $U\subset S$ be an open subset such that $\overline{U}\subset S\setminus S_0$. Then the volume of the tubular neighborhood $U_r=\{p\in\hh^1:\xi_S(p)\in U, d(p,S)<r\}$ can be written as
\begin{equation}
\label{eq:steiner23}
\begin{split}
|U_r|&=A(U)r+\frac{1}{2}\bigg(\int_U HdP\bigg)r^2
\\
&-\frac{1}{6}\bigg(\int_U\bigg(4e_1\bigg(\ntnh\bigg)+(2n+2)\bigg(\ntnh\bigg)^2+|\sg|^2-H^2\bigg)dP\bigg)r^3
\\
&+o(r^4),
\end{split}
\end{equation}
where $dP=\mnh dS$ is the sub-Riemannian area element on $S$.
\end{theorem}

Let us finally look at the case of an umbilic hypersurface

\begin{theorem}
\label{thm:umbilictube}
Let $S\subset\hh^n$, $n\ge 2$, be an umbilic hypersurface of class $C^k$, $k\ge 2$, and let $U\subset S$ be an open subset such that $\overline{U}\subset S\setminus S_0$. Then the volume of the tubular neighborhood $U_r=\{p\in\hh^1:\xi_S(p)\in U, d(p,S)<r\}$ can be written as
\begin{equation}
\label{eq:steinerumb}
|U_r|=\int_U\bigg\{\int_0^r\frac{1}{2}\,(\dot{c}_1c_2-c_1\dot{c}_2)(s)\,\det(D(s))^{n-1}(s)\,ds\bigg\}\,dS,
\end{equation}
where $D$ is the matrix
\[
D=\begin{pmatrix}
1-\mu f_1-\tfrac{\la^2}{2}f_2 & -\tfrac{\la}{2}f_1+\la\mu f_2 \\
\tfrac{\la}{2}f_1-\la\mu f_2 & 1-\mu f_1-\tfrac{\la^2}{2} f_2
\end{pmatrix},
\]
and $\mu$ is the principal curvature of any tangent horizontal vector orthogonal to $J(\nuh)$.
\end{theorem}

\begin{proof}
We use the same notation as in the previous case. Proposition~\ref{prop:cchy} implies that
\begin{equation}
\label{eq:Ae1umb}
A(e_1)=\rho e_1, \quad A(e_i)=\mu e_i,\ i\ge 3,
\end{equation}
and
\begin{equation}
\label{eq:nablalaumb}
\nabla_S^h\bigg(\ntnh\bigg)=e_1\bigg(\ntnh\bigg)\,e_1=\bigg(\mu(\mu-\rho)-\bigg(\ntnh\bigg)^2\bigg)\,e_1.
\end{equation}
In particular,
\[
e_i\bigg(\ntnh\bigg)=0,\quad\text{for all } i\ge 3.
\]
When $i\ge 3$, we have
\begin{equation}
\label{eq:A(ei)}
\nabla_{e_i}\nuh=-\mu e_i-\frac{\la}{2}\,J(e_i).
\end{equation}

We observe first that, when $i\ge 3$, $c_i(0)=\dot{c}_i(0)=\ddot{c}_i(0)=0$ because of equations \eqref{eq:ci'0}, \eqref{eq:A(ei)} and \eqref{eq:ci''0}. Since $e_i(\la)=0$ we get that $c_i(s)\equiv 0$ for $i\ge 3$.

Equations \eqref{eq:explicitUh}, \eqref{eq:Ae1umb}, \eqref{eq:nablaEnuh} and \eqref{eq:nablalaumb} imply that $(E_1)_h$ and $(E_2)_h$ are linear combination of $e_1$ and $J(e_1)$. Hence the scalar products of $E_1$ and $E_2$ with any $X_i$ or $Y_i$, $i\ge 2$, is identically zero.

Taking now any $i\in\{1,\ldots,n\}$ we get from \eqref{eq:explicitUh} and \eqref{eq:A(ei)} that
\begin{align*}
E_{2i-1}&=\big(1-\mu f_1-\tfrac{\la^2}{2}f_2\big)\,e_{2i-1}^\ell+\big(-\tfrac{\la}{2}f_1+\la\mu f_2\big)\,e_{2i}^\ell,
\\
E_{2i}&=\big(\tfrac{\la}{2}f_1-\la\mu f_2\big)\,e_{2i-1}^\ell+\big(1-\mu f_1-\tfrac{\la^2}{2}f_2\big)\,e_{2i}^\ell.
\end{align*}

Hence the Jacobian matrix \eqref{eq:matrix-2} is of the form
\[
\begin{pmatrix}
C & 0 & 0 & \dots & 0 \\
0 & D & 0 & \ldots & 0 \\
0 & 0 & D & \ldots & 0 \\
\vdots & \vdots & \vdots & \ddots & \vdots \\
0 & 0 & 0 & \ldots & D
\end{pmatrix},
\]
where
\[
C=\begin{pmatrix} \tfrac{1}{2}\dot{c}_1 & c_1 \\
\tfrac{1}{2}\dot{c}_2 & c_2 
\end{pmatrix},\quad 
D=\begin{pmatrix}
1-\mu f_1-\tfrac{\la^2}{2}f_2 & -\tfrac{\la}{2}f_1+\la\mu f_2 \\
\tfrac{\la}{2}f_1-\la\mu f_2 & 1-\mu f_1-\tfrac{\la^2}{2} f_2
\end{pmatrix},
\]
that implies \eqref{eq:steinerumb}.
\end{proof}

As an application, we characterize the surfaces $S\subset\hh^1$ of class $C^2$ such that, for any open set $U\Subset S\setminus S_0$ whose closure is a compact subset of $S\setminus S_0$, the Steiner function $|U_r|$ is a polynomial. We will then prove that $S$ is locally a vertical cylinder. By equation \eqref{eq:tub1} we have
\[
|U_r|= \sum_{i=0}^4\int_U\bigg\{\int_0^r a_i f_i(\la, s)\,ds\bigg\}\,dS.
\]

In case $\la\equiv 0$ formulas \eqref{eq:Fi} and \eqref{eq:fi} imply that $f_i(\la,s)=s^i$ for all $i=0,\ldots,4$. In this case it also follows from \eqref{eq:ai} that $a_2=a_3=a_4=0$. So $|U_r|$ is a degree two polynomial. Observe that the equality $\la\equiv 0$ implies that $S\setminus S_0$ is locally a vertical cylinder since the Reeb vector field $T$ is tangent to $S$.

Hence assume that $|U_r|$ is a polynomial for any subset $U\Subset S\setminus S_0$. Then the function $\sum_{i=0}^4 a_if_i(\la,s)$ is a polynomial at every point in $U$. Let us prove that $\la=0$ at any given point reasoning by contradiction. So assume that $\la =0$ at a given point. Straightforward computations show that the series expansion of the functions $f_i(\la,s)$, when $\la\neq 0$ and $i=0,\ldots,4$, are given by 
\begin{align*}
f_0(\la,s)&=\sum_{k=0}^\infty \frac{(-1)^k}{(2k)!}\,(\la s)^{2k},
\\
f_1(\la,s)&=\frac{1}{\la}\sum_{k=0}^\infty \frac{(-1)^k}{(2k+1)!}\,(\la s)^{2k+1},
\\
f_2(\la,s)&=\frac{1}{\la^2}\sum_{k=1}^\infty \frac{(-1)^{k+1}}{(2k)!}\,(\la s)^{2k},
\\
f_3(\la,s)&=\frac{1}{\la^3}\sum_{k=0}^\infty (-1)^{k+1}\frac{2k}{(2k+1)!}\,(\la s)^{2k+1},
\\
f_4(\la,s)&=\frac{1}{\la^4}\sum_{k=1}^\infty (-1)^k\frac{2(k-1)}{(2k)!}\,(\la s)^{2k}.
\end{align*}
Observe that, in the expansions of $f_0,f_2,f_4$ only even terms appear, while in the expressions of $f_1,f_3$ only odd terms appear. Hence we have
\begin{align*}
\sum_{i=0}^4 a_i f_i(\la,s)=a_0&+\sum_{k=1}^\infty\frac{(-1)^k}{(2k)!}\frac{1}{\la^4}\bigg(\la^4a_0-\la^2a_2+2(k-1)a_4\big)\,(\la s)^{2k}
\\
&+\sum_{k=0}^\infty \frac{(-1)^k}{(2k+1)!}\frac{1}{\la^3}\big(\la^2a_1-2ka_3\big)\,(\la s)^{2k+1}.
\end{align*}
In case this function is a polynomial we have
\begin{align*}
\la^4a_0-\la^2a_2+2(k-1)a_4&=0,
\\
\la^2a_1-2k a_3&=0,
\end{align*}
for all $k$ large enough. This implies $a_3=a_4=0$. From these equalities and the expressions \eqref{eq:ai} we get $e_1(\la)=e_2(\la)=0$ and so, again from \eqref{eq:ai}, we obtain~$a_2=0$. It follows from the first equation that $\la^2 a_0=a_2=0$. But this is a contradiction to our assumption $\la\neq 0$ since $a_0\neq 0$.

\bibliography{pr}

\end{document}